\documentclass{amsart} 

\usepackage[english]{babel}
\usepackage[latin1]{inputenc} 
\usepackage{amsmath}
\usepackage{amsthm}
\usepackage{amssymb}
\usepackage{tikz}
\usepackage{graphicx} 

\usepackage{imakeidx}
\usepackage{appendix}
\usepackage{tikz-cd}
\usepackage{verbatim}
\usepackage{enumitem}
\usepackage{multicol}

\usepackage{hyperref}
\usepackage{cleveref}
\usepackage{url}
\usepackage[extra]{tipa}

\newtheorem{thm}{Theorem}[section] 
\newtheorem{prop}[thm]{Proposition}

\newtheorem{lemma}[thm]{Lemma}
\newtheorem{cor}[thm]{Corollary}

\theoremstyle{definition}
\newtheorem{defin}[thm]{Definition}
\newtheorem{example}[thm]{Example}
\theoremstyle{remark}
\newtheorem{rmk}[thm]{Remark}

\numberwithin{equation}{section} 

\newcommand{\R}{\mathbb R}

\newcommand{\Q}{\mathbb Q}

\newcommand{\Z}{\mathbb Z}

\renewcommand{\c}{\subseteq}

\renewcommand{\O}{\mathcal O}

\newcommand{\N}{\mathbb N}
\newcommand{\A}{\mathbb A}

\newcommand{\mc}[1]{\mathcal{#1}}
\newcommand{\cl}{\overline}
\newcommand{\set}[1]{\{#1\}}

\renewcommand{\phi}{\varphi}
\newcommand{\on}[1]{\operatorname{#1}}

\DeclareMathOperator{\Spec}{Spec}

\DeclareMathOperator{\coker}{coker}

\DeclareMathOperator{\ed}{ed}

\DeclareMathOperator{\trdeg}{trdeg}

\newcommand{\ang}[1]{\left \langle{#1}\right \rangle}
\newcommand{\floor}[1]{\left \lfloor{#1}\right \rfloor}

\makeatletter
\newcommand{\doublewidetilde}[1]{{%
		\mathpalette\double@widetilde{#1}%
}}
\newcommand{\double@widetilde}[2]{%
	\sbox\z@{$\m@th#1\widetilde{#2}$}%
	\ht\z@=.9\ht\z@
	\widetilde{\box\z@}%
}
\makeatother

\title{Essential dimension of representations of algebras}
\author{Federico Scavia}

\begin{document}

	\begin{abstract}
Let $k$ be a field, $A$ be a finitely generated associative $k$-algebra and
$\on{Rep}_A[n]$ be the functor $\on{Fields}_k\to \on{Sets}$, which sends a field $K$ containing $k$ to the set of isomorphism classes of representations of $A_K$ of dimension at most $n$.
We study the asymptotic behavior of the essential dimension of this
functor, i.e., the function
$r_A(n) := \ed_k(\on{Rep}_A[n])$, as $n\to\infty$. In particular, we show that the
rate of growth of $r_A(n)$
determines the representation type of $A$. That is, $r_A(n)$ is bounded from
above if $A$ is of finite
representation type, grows linearly if $A$ is of tame representation type,
and grows quadratically if $A$ is
of wild representation type. Moreover, $r_A(n)$ allows us to construct invariants of algebras 
which are finer than the representation type.
	\end{abstract}
	\maketitle
	\section{Introduction}
	Let $k$ be an algebraically closed field, and let $A$ be a finitely generated $k$-algebra (associative, unital, but not necessarily commutative). We begin by recalling the notion of
representation type of $A$, due to Yu. Drozd. We will use the terms ``module" and ``representation" interchangeably.
	
	The algebra $A$ is of \emph{finite representation type} if there are only finitely many indecomposable finite-dimensional $A$-modules, up to isomorphism. For example, if $A=kG$ is a group algebra for a finite group $G$ and $\on{char}k=0$, then $A$ is of finite representation type. 
	
	Loosely speaking, $A$ is \emph{tame} if it admits infinitely many indecomposable representations and if for each $n\geq 0$ the indecomposable $A$-modules occur in a finite number of one-parameter families. The main example is the polynomial algebra $A=k[t]$: the indecomposable $n$-dimensional representations of $A$ correspond to Jordan blocks of size $n$, and the parameter is the eigenvalue.
	
	Finally, $A$ is \emph{wild} if a subset of the isomorphism classes of indecomposable $A$-modules can be parametrized in a one-to-one manner using the indecomposable representations of the free algebra $k\set{x,y}$ on $2$ generators. We refer the reader to \Cref{prelimalg} for the precise definitions.
	
	A classification of the representations of $k\set{x,y}$, in the spirit of those for group algebras in characteristic zero or $k[t]$, is considered to be hopeless; see \cite{ringel1974representation}. Roughly speaking, when $A$ is of finite representation type or tame one can explicitly classify its representations, and when $A$ is wild such a classification is impossible.
	
	When first confronted with these definitions, one may be surprised by the big gap between the notions of tame and wild. However, when $A$ is finite-dimensional, there are no intermediate possibilities. According to a celebrated theorem of Drozd \cite{drozd}, $A$ is of exactly one of the three
    representation types we described: finite, tame or wild; see \cite[Theorem 1, Proposition 2, Corollary 1]{drozd} or \cite[Theorem B]{crawley1988tame}.
	
	The purpose of the present work is to reinterpret and refine Drozd's Theorem via \emph{essential dimension}. We denote by $r_A(n)$ the essential dimension of the functor of representations of $A$ of dimension at most $n$. By definition, $r_A(n)$ is the smallest integer $m\geq 0$ such that for every field extension $K/k$ and every representation $M$ of $A_K=A\otimes_kK$ such that $\dim_KM\leq n$, there exist a subfield $k\c K_0\c K$ such that $\trdeg_kK_0\leq m$ and a representation $N$ of $A_{K_0}$ such that $N\otimes_{K_0}K\cong M$; see \Cref{essdim} for further details.
	
	The definition of $r_A(n)$ takes as input an enormous amount of information: we are considering all $A_K$-representations for every field extension $K/k$. In particular, even in the case where $k$ is algebraically closed, we are forced to consider representations over fields that are not necessarily algebraically closed. 
		
	The main result of this paper is the following refinement of Drozd's Theorem. It follows from the combination of \Cref{trichotomy1}, \Cref{trichotomy2} and \Cref{trichotomy3}. 
	
	\begin{thm}\label{trichotomy-finite}
	Let $A$ be a finite-dimensional algebra over an algebraically closed field $k$.
		\begin{enumerate}[label=(\alph*)]
			\item If $A$ is of finite representation type, then \[r_A(n)=0\] for every $n\geq 1$.
			\item If $A$ is tame, then there exists $c>0$ such that \[cn-1\leq r_A(n)\leq 2n-1\] for every $n\geq 1$. 
			\item If $A$ is wild, then there exists $c>0$ such that \[r_A(n)\geq cn^2-1\] for every $n\geq 1$.
		\end{enumerate}	
	\end{thm}

Some remarks are in order. 
\begin{enumerate}[label=(\roman*)]
	\item \Cref{trichotomy-finite} gives a common framework for several seemingly unrelated results of \cite{karpenko2014numerical} and \cite{bensonreichstein}; see \Cref{groupalgebras} for further details. 
    \item Part (a) of \Cref{trichotomy-finite} is \cite[Theorem 1.3]{bensonreichstein}. However, if $k$ is only assumed to be perfect, we will show that $r_A(n)$ is bounded from above; see \Cref{trichotomy1}. Parts (b) and (c) of \Cref{trichotomy-finite} hold when $k$ is arbitrary, $A$ is finitely generated (not necessarily finite-dimensional) and $A_{\cl{k}}$ is tame or wild, respectively; see \Cref{trichotomy2} and \Cref{trichotomy3}. 
    \item If $k$ is not algebraically closed, the representation type of $k$-algebras becomes more subtle to define; see \Cref{variants}. 
	\item If $A$ is of tame or wild type, it is still possible that $r_A(n)=0$ for small values of $n$. This explains the presence of $-1$ in the lower bounds.
	\item If $A$ is generated by $r$ elements over
	$k$, then every $A_K$-module $M$ is defined over the subfield $K_0$ of $K$ generated over $k$ by the $rn^2$ matrix entries of left multiplication by the generators. Thus we have the following naive upper bound
	\[\ed_kM\leq\on{trdeg}_k(K_0)\leq rn^2\] which shows that quadratic growth is the fastest possible.
	\item 	Our proof of \Cref{trichotomy-finite} is based on combining
	stack-theoretic techniques with representation-theoretic arguments. The
	stack-theoretic techniques we use were initially developed in \cite{biswasdhillonhoffmann}, for
	the purpose of computing the essential dimension of the stack of vector bundles on
	a given curve. In this paper we modify these techniques and adapt them to study the
	essential dimension of representations of algebras.  Some of our
	representation-theoretic arguments make use of results from logic and
	model theory \cite{jensen1982homological,kasjanperiodicity}.
\end{enumerate}

	When $k$ is algebraically closed and $A$ is finite-dimensional, \Cref{trichotomy-finite} tells us that the asymptotic behavior of $r_A(n)$ determines the representation type of $A$. We may then regard this function as a finer invariant of $A$, and use it to extract numerical invariants. For every field $k$ and every finitely generated $k$-algebra $A$, set:
	\begin{align*}
	a_0(A):=\lim_{n\to\infty}r_A(n),\qquad &\text{if $A_{\cl{k}}$ is of finite representation type,}\\
	a_1^+(A):=\limsup_{n\to\infty}\frac{r_A(n)}{n},\qquad &\text{if $A_{\cl{k}}$ is tame,}\\
	a_2^+(A):=\limsup_{n\to\infty}\frac{r_A(n)}{n^2},\qquad &\text{if $A_{\cl{k}}$ is wild.}
	\end{align*}
Using $\liminf_{n\to\infty}$ instead, one may also define $a_1^-(A),a_2^-(A)$. We also write
	\begin{align*}
a_1(A):=\lim_{n\to\infty}\frac{r_A(n)}{n},\qquad &\text{if $A_{\cl{k}}$ is tame,}\\
a_2(A):=\lim_{n\to\infty}\frac{r_A(n)}{n^2},\qquad &\text{if $A_{\cl{k}}$ is wild,}
\end{align*}
when such limits exist. 	
When $k$ is algebraically closed and $A$ is finite-dimensional, \Cref{trichotomy-finite} shows that if $A$ is tame then $0<a_1^{-}(A)\leq a_1^+(A)\leq 2$, and that if $A$ is wild then $a_2^{-}(A)>0$.
	
The number $a_0(A)$ has been studied in \cite{karpenko2014numerical} and \cite{bensonreichstein}. When they exist, the numbers $a_1(A)$ and $a_2(A)$ represent the coefficients of the ``leading term" of $r_A(n)$, as $n\to\infty$. It may also be of interest, even though beyond the scope of this paper, to investigate the ``next term", i.e., the rate of growth of $r_A(n) - a_1(A) n$ for tame algebras, or $r_A(n) - a_2(A) n^2$ for wild algebras. 

To demonstrate that the invariants $a_i^{\pm} (A)$ ($i = 0, 1,
2$) are accessible, at least in some cases, we will compute them
explicitly in the case, where $A$ is a quiver algebra. Let $Q$ be a quiver, and let $A=kQ$ be its path algebra; see \Cref{prelimquiver} for definitions and references. The algebra $kQ$ is finitely generated, and is finitely dimensional if and only if $Q$ has no oriented cycles. The representation type of $Q$ is, by definition, the representation type of $\cl{k}Q$. Gabriel's Theorem \cite{gabriel} states that $Q$ is of finite representation type if and only if its underlying graph is a Dynkin diagram of type $A,D,E$; see also \cite[Theorem 3.3]{kirillov} or \cite[Theorem 8.12]{schiffler}. The quiver $Q$ is tame if its underlying graph is an extended Dynkin diagram of type $\widetilde{A},\widetilde{D},\widetilde{E}$, and it is wild in the remaining cases. In particular, every path algebra $\cl{k}Q$ is of finite, tame or wild representation type, as in
Drozd's Theorem, even though such algebras are allowed to be infinite-dimensional. We will sometimes collectively refer to quivers of finite or tame representation type as non-wild quivers.

If $Q$ is wild, let $\Lambda_Q$ be the maximum of the opposite of the Tits form of $Q$ on \[\set{\alpha\in\R^{Q_0}_{\geq 0}: \sum_{i\in Q_0}\alpha_i=1}.\] Here $Q_0$ is the set of vertices of $Q$. As we explain in \Cref{maximum}, $\Lambda_Q$ can be easily computed from the underlying graph of $Q$; we give several examples in \Cref{threefamilies}.

The following theorem follows from the combination of \Cref{finrep}(b), \Cref{tameed} and \Cref{wild}.
\begin{thm}\label{wild-intro}
	Let $k$ be an arbitrary field, and let $Q$ be a connected quiver (possibly with loops and oriented cycles).
		\begin{enumerate}[label=(\alph*)]
	    \item If $Q$ is of finite representation type, then \[r_{kQ}(n)=0\] for every $n\geq 1$.
		\item If $Q$ is tame and $\delta=(\delta_i)_{i\in Q_0}$ is the null root of $Q$, then \[r_{kQ}(n)=\floor{\frac{n}{\sum\delta_i}}\] for every $n\geq 1$, where the sum is over the set of vertices of $Q$.
		\item If $Q$ is wild, then \[a_2(kQ)=\Lambda_Q;\] in particular $a_2(kQ)\in\Q_{>0}$. Moreover, $a_2(kQ)\geq \frac{1}{2480}$, with equality if and only if $Q$ is the disjoint union of a (possibly empty) non-wild quiver and of quivers of type $\widetilde{\widetilde{E_8}}$
		\[
		\begin{tikzcd}
		&& 4 \\
		1 \arrow[r,-]
		& 2 \arrow[r,-]
		& 3 \arrow[r,-] \arrow[u,-]
		& 5 \arrow[r,-]
		& 6 \arrow[r,-]
		& 7 \arrow[r,-]
		& 8 \arrow[r,-]
		& 9 \arrow[r,-]
		& 10.
		\end{tikzcd}
		\]\noindent
	\end{enumerate}
\end{thm}
Part (a) of \Cref{wild-intro} follows from results of Kac and Schofield that predate essential dimension; see \Cref{finrep}. The proof of part (b) rests on the classification of representations of tame quivers over an arbitrary field. We refer the reader to
\Cref{prelimquiver} for the definition of the null root. The proof of part (c) relies on
stack-theoretic techniques. In particular, it is crucial that the algebraic stack parametrizing representations of $Q$ is smooth over $k$.

It follows from \Cref{wild-intro} that the limits $a_1(kQ)$ and $a_2(kQ)$ exist when $Q$ is tame or wild, respectively. The existence of the limits $a_1(A)$ and $a_2(A)$ for an arbitrary $k$-algebra $A$ is an open problem.

It has been brought to our attention by Richard Lyons that $2480$ is the dimension of the minimal faithful complex representation of the Lyons group. The Lyons group is one of the $26$ sporadic finite simple groups; see \cite{lyons1972evidence} and \cite{sims1973existence}.  Why this particular number appears in Theorem 1.2(c) is a bit of a mystery.
	
	\subsection*{Notational conventions}
	Throughout this paper $k$ will denote a fixed base field, and $A$ a finitely generated associative unital $k$-algebra. We will denote by $\cl{k}$ an algebraic closure of $k$. For a field extension $K/k$, we will denote by $A_K$ the tensor product $A\otimes_kK$. When we consider an $A_K$-module $M$, unless otherwise specified we will assume that $M$ is a finite-dimensional $K$-vector space. For a field extension $L/K$, we will denote $M\otimes_KL$ by $M_L$.
	
	If $R$ is a $k$-algebra, we denote by $j(R)$ its Jacobson radical.
	
	
	\section{Preliminaries on essential dimension}\label{essdim}
	The definition of $r_A(n)$ is a special case of essential dimension of functors and stacks. We start by giving the definition of the essential dimension, due to Merkurjev \cite{berhuy2005notion} in the context of functors, and to Brosnan, Reichstein and Vistoli \cite{genericity0} for algebraic stacks. 
	
	\begin{defin}
	Let $\on{Fields}_k$ denote the category of field extensions of $k$. Let $F:\on{Fields}_k\to\on{Sets}$ be a functor.
	\begin{enumerate}[label=(\roman*)]
	    \item An element $\xi\in F(L)$ is \emph{defined over a field} $K\c L$ if it belongs to the image of $F(K)\to F(L)$.
	    \item The \emph{essential dimension} of $\xi\in F(L)$ is \[\ed_k\xi:=\min_K\trdeg_kK,\] where the minimum is taken over all fields of definition $K$ of $\xi$.
	    \item The \emph{essential dimension} of the functor $F$ is defined to be
	\[\ed_kF:=\sup_{(K,\xi)}\ed_k\xi,\] where the supremum is taken over all pairs $(K,\xi)$, where $K$ is a field extension of $k$, and $\xi\in F(K)$. 
	\item If $\mc{X}$ is an algebraic stack over $k$, we obtain a functor	$F_{\mc{X}}:\on{Fields}_k\to\on{Sets}$ sending a field $K$ containing $k$ to the set of isomorphism classes of objects in $\mc{X}(K)$. We define the essential dimension of an object $\eta\in\mc{X}(K)$ as the essential dimension of its isomorphism class in $F_{\mc{X}}(K)$, and the essential dimension of $\mc{X}$ as $\ed_kF_{\mc{X}}$.
	\end{enumerate} 
	\end{defin}
	
	Consider the functor $\on{Rep}_A[n]:\on{Fields}_k\to \on{Sets}$ given by \[\on{Rep}_A[n](K):=\set{\text{$K$-isomorphism classes of $A_K$-modules of dimension $\leq n$}}\] for every field extension $K/k$, and such that for every inclusion $K\c L$, the corresponding map $\on{Rep}_A[n](K)\to \on{Rep}_A[n](L)$ is induced by tensor product. For every $n\geq 1$, we define \[r_A(n):=\on{ed}_k\on{Rep}_A[n].\]
    
    We can do more: $r_A(n)$ is the essential dimension of an algebraic stack over $k$. Since stack-theoretic methods are central to this work, we explain this construction in detail. We start by choosing a presentation of $A$ as a quotient of a finitely generated free algebra \[A=k\set{x_1,\dots,x_r}/I,\] where $I$ is a two-sided ideal of $A$. We denote by $a_1,\dots,a_r$ the images of $x_1,\dots,x_r$ in $A$. 

	For every $d\geq 0$, consider the affine space \[X_d:=\prod_{i=1}^r\on{M}_{d\times d,k}.\] The group $\on{GL}_d$ acts on $X_d$ by simultaneous conjugation.
	
	Let $K/k$ be a field extension, and let $M$ be a $d$-dimensional $A_K$-module. By fixing a $K$-basis for $M$, left multiplication by $a_1,\dots,a_r$ gives rise to $d\times d$ matrices with entries in $K$, yielding a $K$-point $\alpha=(\alpha_1,\dots,\alpha_r)$ of $X_d$. If we choose a different basis for $M$, and $g\in \on{GL}_d(K)$ is the matrix of this base change, the new $K$-point associated to $M$ will be $g\cdot\alpha$. Moreover, 
	\begin{equation}\label{poleq}
	P(\alpha_1,\dots,\alpha_r)=0 \text{ for each $P\in I$}.
	\end{equation}
	Let $Y_d$ be the closed $\on{GL}_d$-invariant subscheme of $X_d$ defined by the polynomial equations of (\ref{poleq}). We can form the stacks \[\mc{R}_A[n]:=\coprod_{d=1}^n[Y_d/\on{GL}_d],\qquad \mc{R}_A:=\coprod_{d=1}^{\infty}[Y_d/\on{GL}_d],\] where $[Y_d/\on{GL}_d]$ denotes the quotient stack construction. The algebraic stacks $\mc{R}_A[n]$ are of finite type over $k$, and $\mc{R}_A$ is locally of finite type over $k$; they are not necessarily smooth.  We claim that \[\ed_k\mc{R}_A[n]=r_A(n).\] 
	
	We have just seen that the $\on{GL}_d(K)$-orbits in $Y_d(K)$ bijectively correspond to the isomorphism classes of $d$-dimensional $A_K$-modules. By \cite[Example 2.6]{genericity0}, the $K$-points of $[Y_d/\on{GL}_d]$ bijectively correspond to the $\on{GL}_d(K)$-orbits of $Y_d(K)$. We have thus constructed a natural bijection between $K$-points of $\mc{R}_A[n]$ and isomorphism classes of $A_K$-modules of dimension at most $n$, that is, the functors $F_{\mc{R}_A[n]}$ and $\on{Rep}_A[n]$ are naturally isomorphic. In particular, $\ed_k\mc{R}_A[n]=r_A(n)$, as claimed.
	
	It is easy to see that $\mc{R}_A[n]$ is independent of the choice of the generators of $A$, up to isomorphism.
	
	We conclude this section with the following observation, which will be used during the proof of \Cref{trichotomy-finite}.
	\begin{lemma}\label{rosen}
		Assume that $k$ is algebraically closed. Let $G$ be a connected algebraic group over $k$, and let $H\c G$ be a closed subgroup, either finite or connected. Let $X$ be a $G$-variety (not necessarily irreducible), and let $Y$ be an irreducible $H$-variety. Assume that there exists an $H$-equivariant rational map $f:Y\to X$ such that for any $G$-orbit in $X$ only finitely many $H$-orbits of $Y$ are mapped to it. Then \[\ed_k[X/G]\geq \trdeg_kk(Y)^H.\]
	\end{lemma}
	
	\begin{proof}
		Since $G$ is connected, every irreducible component of $X$ is $G$-stable. Since $Y$ is irreducible, there exists a component $X_0$ of $X$ such that $\cl{f(Y)}\c X_0$. We clearly have $\ed_k[X/G]\geq \ed_k[X_0/G]$, hence we may assume that $X=X_0$ is irreducible.
		
		We may find invariant open subsets $V\c Y$ and $U\c X$ such that $f(V)\c U$ and such that there exist geometric quotients $V/H$ and $U/G$. This follows from an application of Rosenlicht's Theorem \cite[Theorem 2]{rosenlicht1956} (the quoted result of Rosenlicht only applies to connected groups, but it is well known that the quotient $Y/H$ exists if $H$ is finite).
		
		The induced morphism $V/H\to U/G$ has generically finite fibers by assumption. By the fiber dimension theorem \[\dim U/G\geq \dim V/H=\trdeg_kk(V)^H.\] The projection $[U/G]\to U/G$ is surjective, so there exist a field extension $K/k$ and a $K$-point $\xi$ of $[X/G]$ mapping to the generic point of $U/G$. Then \[\ed_k\xi\geq\dim U/G\geq \trdeg_kk(V)^H.\qedhere\]
	\end{proof}	
	
	\section{Representation types}\label{prelimalg}
	Let $k$ be an arbitrary field, and let $A$ be a finitely generated $k$-algebra. In this section we define the representation type of $A_{\cl{k}}$ and give some examples. The following definitions are due to Drozd \cite{drozd}.
		
	\begin{defin}
		Let $\Lambda$ be a $k$-algebra. A $\Lambda$-\emph{representation} of $A$ is an $A-\Lambda$-bimodule $N$, that is free of finite rank as a right $\Lambda$-module. 
		
		 We say that $N$ is \emph{strict} if for each pair of $\Lambda$-modules $M$ and $M'$ such that $N\otimes_{\Lambda} M\cong N\otimes_\Lambda M'$ as $A$-modules one has $M\cong M'$ as $\Lambda$-modules. 
		 One may also think of $N$ as a functor, see \cite[\S 2]{crawley1988tame}.		
	\end{defin}	

\begin{defin}\label{reptype}
	Let $k$ be an algebraically closed field and let $A$ be a finitely generated $k$-algebra.
	\begin{itemize}
		\item We say that $A$ is \emph{of finite representation type} if there are at most finitely many isomorphism classes of indecomposable $A$-modules.
		\item The algebra $A$ is \emph{tame} if it is not of finite representation type and if, for every positive integer $d$, there exists a finitely generated $k$-algebra of the form $k[x]\c\Lambda\c k(x)$, together with a finite collection $\set{N_j}$ of $\Lambda$-representations of $A$, such that any $d$-dimensional indecomposable representation of $A$ is isomorphic to $N_j\otimes_{\Lambda}M$ for some $j$ and some $1$-dimensional $\Lambda$-module $M$.
		\item We call $A$ \emph{wild} if there exists a strict $k\set{x,y}$-representation $N$ of $A$ such that, for every $k\set{x,y}$-module $M$, the representation $N\otimes_{k\set{x,y}}M$ is indecomposable.
	\end{itemize}	
	If $k$ is an arbitrary field, we say that $A$ is \emph{of finite representation type} if there are at most finitely many isomorphism classes of indecomposable $A$-modules. 
\end{defin}
	
	Since in \Cref{trichotomy-finite} the field $k$ is algebraically closed, \Cref{reptype} is sufficient to understand the statement of the theorem. It seems natural to prove analogous results for an algebra $A$ over a more general field, and we do so in \Cref{proof}. In \Cref{trichotomy1}, we will only assume that $k$ is perfect and that $A$ is finite-dimensional and of finite representation type. In \Cref{trichotomy2} and \Cref{trichotomy3}, the minimal assumption to make our argument work is that $A_{\cl{k}}$ is tame or wild, respectively. Therefore, we do not need to introduce more subtle notions of tameness and wildness over arbitrary fields. Nevertheless, see \Cref{variants} for a discussion of possible variants of the definitions.
	
	\begin{example}\label{nilpotent}
	Let $m\geq 1$. If $A=k[x]/(x^m)$, $n$-dimensional $A_K$-modules correspond to conjugacy classes of $K$-linear endomorphisms having index of nilpotency at most $m$. The indecomposable representations correspond to nilpotent Jordan blocks of size at most $m$, and these are all defined over the base field $k$. Therefore $r_A(n)=0$. The algebra $A$ is of finite representation type.
\end{example}

\begin{example}\label{quaternion}
	Let $k=\Q$, and let \[A=\Q\set{i,j}/(i^2=j^2=-1, ij=-ji)\] be the quaternion algebra over $\Q$. Since $A$ is a group algebra over a field of characteristic zero, it is of finite representation type. Let $K$ be the field of fractions of $\Q[a,b]/(a^2+b^2+1)$, and let $M$ be the $2$-dimensional $A_K$-module given by 
	\begin{equation*}
	i\mapsto \begin{pmatrix}
	a & -b \\
	b & a
	\end{pmatrix},\qquad
	j\mapsto \begin{pmatrix}
	b & -a \\
	-a & -b
	\end{pmatrix}.
	\end{equation*}
	In \cite[Proposition 6.3]{bensonreichstein}, it is shown that $\ed_kM=1$.
	\end{example}

	\begin{example}\label{basicexample}
	Let $A=k[t]$. Isomorphism classes of $A_K$-modules correspond to conjugacy classes of $K$-linear endomorphisms of $K^n$. These are classified by the rational canonical form, hence $r_A(n)=n$ (we refer the reader to \cite{reichstein2010essential} for the details). The algebra $A_{\cl{k}}$ is the prototypical example of an algebra of tame representation type.
\end{example}

	\begin{example}\label{polynomial}
	Let $A=k[x,y]$ be a polynomial algebra in two variables. Representations of $A_K$ correspond to pairs of commuting matrices with entries in $K$. In \cite[Lemma 1]{drozd1972representations}, a strict $\cl{k}\set{x,y}$-representation of $A$ of rank $32$ is given. This shows that $A_{\cl{k}}$ is wild, so according to \Cref{trichotomy-finite}(c), $r_A(n)$ grows quadratically in $n$. 
	\end{example}
	
	\begin{example}
	Let $A:=k\set{x,y,z}/I$, where $I$ is the ideal generated by all monomials of degree $2$ in $x,y,z$. Then $A_{\cl{k}}$ is an example of a finite-dimensional wild algebra; see \cite[(1.2)]{ringel1974representation}.
	\end{example}

	\section{Fields of definition for representations}
	In this section we adapt the methods of \cite[\S 5]{biswasdhillonhoffmann} to the setting of representations of algebras.	
	Let $A$ be a finitely generated $k$-algebra, let $K$ be a field containing $k$, and let $M$ be an $A_K$-module. We may view $M$ as a $K$-point of $\mc{R}_A$, as explained in \Cref{essdim}. We may then associate to $M$ its \emph{residue gerbe} $\mc{G}$ in $\mc{R}_A$, and its \emph{residue field} $k(M):=k(\mc{G})$; we refer the reader to \cite[Chapitre 11]{laumonmoretbailly} for the definitions. 
	
	\begin{lemma}\label{arithmetic-geometric}
	   We have \[\ed_kM=\ed_{k(M)}M+\trdeg_kk(M).\] 
	\end{lemma}
	\begin{proof}
	By construction, the field $k(M)$ is contained in any field of definition for $M$. Let $K_0$ be a field of definition for $M$ (viewed as an element of $\on{Rep}_{A}[n](K)$) such that $\ed_kM=\trdeg_kK_0$. Then $k(M)\c K_0$ and $K_0$ is a field of definition for $M$ (viewed as an element of $\on{Rep}_{A_{k(M)}}[n](K)$). We deduce that $\trdeg_{k(M)}K_0\geq \ed_{k(M)}M$, hence \[\ed_kM=\trdeg_kK_0=\trdeg_kk(M)+\trdeg_{k(M)}K_0\geq \trdeg_kk(M)+\ed_{k(M)}M.\] 
	
	On the other hand, if $K_1$ is a field of definition for $M$ (viewed as an element of $\on{Rep}_{A_{k(M)}}[n](K)$) such that $\trdeg_{k(M)}K_1=\ed_{k(M)}M$, then $K_1$ is also a field of definition for $M$ (viewed as an element of $\on{Rep}_{A}[n](K)$), hence \[\ed_kM\leq \trdeg_kK_1=\trdeg_kk(M)+\trdeg_{k(M)}K_1=\trdeg_kk(M)+\ed_{k(M)}M.\qedhere\]
	\end{proof}

	\begin{rmk}
	Let $K/k$ be a field extension, and let $M$ be an $A_K$-module. Roughly speaking, the residue gerbe of $M$ parametrizes all pairs $(L,N)$, where $L/k$ is a field extension, $N$ is an $A_L$-module, and $M_E\cong N_E$ for some extension $E/k$ containing $K$ and $L$. 
	
	By construction, $k(M)$ is contained in every field of definition of $M$. Moreover, since $k(M)$ depends only on the residue gerbe of $M$, for every field extension $L/K$ we have $k(M)=k(M_L)$.
		
	Since $k(M)$ is contained in all fields of definition of $M$, it is clear that if $M$ is defined over $k(M)$, then $k(M)$ is the minimal field of definition for $M$. However, it is not always the case that $k(M)$ is a field of definition for $M$. It may even happen that a minimal field of definition does not exist (but see \Cref{thm11bensonreichst} below for a positive result). 
	
	As an example, let $k=\Q$, and take $A$, $K$ and $M$ as in \Cref{quaternion}. By \cite[Proposition 6.3]{bensonreichstein}, $M$ does not have a minimal field of definition. Since $A_{\cl{K}}\cong M_{2\times 2}(\cl{K})$, the module $M_{\cl{K}}$ is the $2$-dimensional vector representation of $M_{2\times 2}(\cl{K})$, which is defined over $\Q$. It follows that $\Q(M)=\Q(M_{\cl{K}})=\Q$.	
	\end{rmk}

	The residue gerbe $\mc{G}$ of $M$ is clearly non-empty. By \cite[Th\'eor\`eme 11.3]{laumonmoretbailly}, $\mc{G}$ is an algebraic stack of finite type over $k(\mc{G})$. Therefore, there exists a (smooth surjective) morphism $U\to \mc{G}$, where $U$ is an algebraic space of finite type over $k(\mc{G})$. Passing to an fppf cover of $U$ if necessary, we may assume that $U$ is a scheme. Since $\mc{G}$ is non-empty, $U$ is non-empty. The Nullstellensatz then guarantees the existence of a finite field extension $l/k(\mc{G})$ for which $U(l)\neq\emptyset$, hence such that $\mc{G}(l)\neq \emptyset$. We let \[d:=[l:k(\mc{G})]<\infty.\] We choose an object $V\in\mc{G}(l)$, and set  \[R:=\on{End}_{k(\mc{G})}(\cl{V})\] where $\cl{V}$ denotes the representation of $M$ over $k(\mc{G})$ obtained from $V$ by restriction of scalars to $k(\mc{G})$. To state the main theorem of this section, we first need to give a definition.
	
	\begin{defin}
		Let $\Lambda$ be a finite-dimensional $k$-algebra. A projective $\Lambda$-module $M$ has \emph{rank} $r\in\Q_{>0}$ if the direct sum $M^{\oplus n}$ is free of rank $nr$ for some $n\in\Z_{>0}$ with $nr\in\Z_{>0}$. We let $\on{Mod}_{\Lambda,r}$ be the category of projective modules of rank $r$.
	\end{defin}
	
	The following result is an analogue of \cite[Theorem 5.3]{biswasdhillonhoffmann} for representations of algebras.
	
	\begin{thm}\label{secondgerbe}
		In the above situation, consider a field $K\supseteq k(\mc{G})$. Then $\mc{G}(K)$ is equivalent to the category of projective right $R_K$-modules of rank $1/d$, compatibly with extension of scalars. In particular, all objects in $\mc{G}(K)$ are isomorphic (Noether-Deuring Theorem), and for the $k(\mc{G})$-algebra $R$ and the integer $d$ defined above, we have
		\[\ed_{k(\mc{G})}\mc{G}=\ed_{k(\mc{G})}(\on{Mod}_{R,1/d})=\ed_{k(\mc{G})}(\on{Mod}_{R/j(R),1/d}).\]
	\end{thm}
	
	\begin{proof}
		The proof is the same as that of \cite[Theorem 5.3, Corollary 5.4]{biswasdhillonhoffmann}.
	\end{proof}
		
	\begin{rmk}\label{generalklinear}
		The same result holds for any algebraic stack $\mc{X}$ over $k$ whose restriction to $\on{Fields}_k^{\on{op}}$ is a $k$-linear fibered category such that for every finite extension $L/K$ the pullback functor $\mc{X}(K)\to\mc{X}(L)$ admits a right adjoint. In the case of vector bundles on a curve, the right adjoint is given by the pushforward of a bundle, and in the case of $A$-modules it is given by restriction of scalars. The proof in this more general setting is again identical to that of \cite[Theorem 5.3]{biswasdhillonhoffmann}. 
	\end{rmk}
	
	The next result is \cite[Theorem 1.1]{bensonreichstein}, without any separability assumption. We will not need this result in the sequel. For the definition of fields of dimension $\leq 1$, see \cite[\S II.3]{serre1997galois} or \cite[p. 2]{bensonreichstein}.
	
	\begin{cor}\label{thm11bensonreichst}
		Assume that $k$ is a field of dimension $\leq 1$ (for example, a $C_1$-field). Let $M$ be an $A_K$-module, where $K$ is an algebraic extension of $k$. Then $M$ has a minimal field of definition $k\c F\c K$, of finite degree over $k$. In fact, $F=k(M)$.
	\end{cor}
	
	\begin{proof}
		We keep the notation of \Cref{secondgerbe}.	Take for $\mc{G}$ the residue gerbe of $M$. The field extension $k(\mc{G})/k$ is finitely generated. It is also algebraic, since $k(\mc{G})\c K$. Hence $k(\mc{G})$ is a finite extension of $k$. Since $k$ has dimension $\leq 1$, $k(\mc{G})$ also has dimension $\leq 1$. Therefore \[R/j(R)=\prod_{i}\on{M}_{n_i\times n_i}(L_i),\] where the $L_i$ are finite field extensions of $k(\mc{G})$; see \cite[\S II.3, Proposition 5]{serre1997galois}. Since $M$ is defined over $K$, every integer $n_i$ is divisible by $d$, so there exists an $R/j(R)$-module of rank $1/d$. Now \Cref{secondgerbe} implies that $M$ is already defined over $k(M)$.
	\end{proof}
	
	Let $K$ be a field containing $k$, and let $M$ be an $A_K$-module. We use the techniques developed so far to give upper bounds on $\ed_kM$, by estimating $\ed_{k(M)}M$ and $\trdeg_kk(M)$ separately.
	
		\begin{lemma}\label{lemma42}
		Let $M$ be an indecomposable $A_K$-module. Then \[\dim_K\on{End}(M)/j(\on{End}(M))\leq \dim_KM.\]
	\end{lemma}
	
	\begin{proof}
		Since $M$ is indecomposable, the $K$-algebra $\on{End}(M)$ is local, therefore $D=\on{End}(M)/j(\on{End}(M))$ is a division algebra. By Nakayama's lemma, $j(\on{End}(M))M\neq M$, so $M/j(\on{End}(M))M$ is a non-zero left $D$-module. Hence $\dim_KD\leq \dim_KM$.
	\end{proof}	
	
	\begin{lemma}\label{edkm}
		Let $M$ be a non-zero finite-dimensional $A_K$-module. Then \[\ed_{k(M)}M\leq \dim_KM-1.\]
	\end{lemma}
	
	\begin{proof}
	Entirely analogous to \cite[Corollary 5.5]{biswasdhillonhoffmann}, after replacing \cite[Lemma 4.2]{biswasdhillonhoffmann} by \Cref{lemma42}.
	\end{proof}
	
	\begin{lemma}\label{algclosedreduction}
		Let $A$ be a finitely generated $k$-algebra. Assume that for every field extension $K/k$, where $K$ is algebraically closed, and for every indecomposable $A_K$-module $N$, one has \[\on{trdeg}_kk(N)\leq 1.\] Then, for every extension $K/k$ and every non-zero $A_K$-module $M$, \[\on{ed}_kM\leq 2\dim_KM-1.\]
	\end{lemma}
	\begin{proof}
		By \Cref{arithmetic-geometric} and \Cref{edkm}, it suffices to show that \[\trdeg_kk(M)\leq\dim_KM\] for every field extension $K/k$ and every $A_K$-module $M$. Let $M$ be an $A_K$-module, for some field extension $K/k$. Since $k(M_L)=k(M)$ for every field extension $L/K$, we may assume that $K$ is algebraically closed. If we express $M$ as a sum of indecomposable $A_K$-modules $M_j$, we have
		\[\trdeg_kk(M)\leq \sum_j\trdeg_kk(M_j).\]
		By assumption, each summand is either $0$ or $1$, and there are at most $\dim_KM$ terms in the decomposition of $M$. Therefore $\trdeg_kk(M)\leq\dim_KM$, as desired. 
	\end{proof}

	\section{Proof of Theorem \ref{trichotomy-finite}}\label{proof}
	\Cref{trichotomy-finite} will follow from the next three propositions.
	
\begin{prop}\label{trichotomy1}
	Let $k$ be a perfect field and $A$ be a finite-dimensional $k$-algebra. Assume that $A$ is of finite representation type. Then there exists a constant $C$ such that \[r_A(n)\leq C\] for every $n\geq 1$.
\end{prop}
	Recall that, when $k$ is perfect, $A$ is of finite representation type if and only if $A_{\cl{k}}$ is; see \Cref{variants}.
\begin{proof}
	By assumption, there are at most finitely many indecomposable $A$-modules, up to isomorphism; we denote them by $N_1,\dots, N_r$. Let \[m:=\sum_{i=1}^r\dim_kN_i.\]
	
	Let $M$ be an $A_K$-module, for some field extension $K/k$. Fix an algebraic closure $\cl{K}$ of $K$. Since $k\c \cl{K}$, the field $\cl{K}$ contains an algebraic closure $\cl{k}$ of $k$. Since $\cl{k}$ is algebraically closed, the Brauer group of $\cl{k}$ is trivial. Therefore, $A_{\cl{k}}$ being finite-dimensional over $\cl{k}$, the hypotheses of \cite[Theorem 1.3]{bensonreichstein} apply to $F=\cl{k}$, $A_{\cl{k}}$ and $M_{\cl{K}}$. It follows that $M_{\cl{K}}$ is defined over a finite extension of $\cl{k}$, hence over $\cl{k}$. Thus \begin{equation}\label{zerotrdeg}\on{trdeg}_kk(M)=\on{trdeg}_{k}k(M_{\cl{K}})\leq \trdeg_k\cl{k}=0,\end{equation} hence $\on{trdeg}_kk(M)=0$ for every $A_K$-module $M$.
	
	Assume first that $M$ is indecomposable. By \cite[Theorem 3.3]{jensen1982homological}, $M$ is a direct summand of a module of the form $(N_i)_K$, hence (i) $\dim_K M\leq\dim_kN_i\leq m$ and (ii) there are at most $m$ isomorphism classes of indecomposable $A_K$-modules. Using \Cref{arithmetic-geometric}, \Cref{edkm}, (\ref{zerotrdeg}) and (i), we conclude that \begin{equation}\label{trichotomy1-eq}
	\ed_kM=\on{trdeg}_kk(M)+\ed_{k(M)}M\leq \dim_KM-1\leq m-1 
	\end{equation}
	when $M$ is indecomposable.
	
	If $M$ is not assumed to be indecomposable, consider the decomposition of $M$ in indecomposable summands $M_h$. For every $h$, we have shown that there exists a subfield $F_h$ of $K$ such that $\on{trdeg}_kF_h\leq m-1$ and $M_h$ is defined over $F_h$. If we let $F$ be the compositum of all the $F_h$, then by (ii) and (\ref{trichotomy1-eq}) we have $\on{trdeg}_kF\leq m(m-1)$. Since every indecomposable summand of $M$ is defined over $F$, by Noether-Deuring's Theorem $M$ is also defined over $F$, hence $\ed_kM\leq m(m-1)$. Since $M$ was arbitrary, we obtain $r_A(n)\leq m(m-1)$ for every $n\geq 0$.     
\end{proof}
	
	\begin{prop}\label{trichotomy2}
	Let $k$ be an arbitrary field, and let $A$ be a $k$-algebra. Assume that $A_{\cl{k}}$ is tame. Then there exists a constant $c>0$ such that \[cn-1\leq r_A(n)\leq 2n-1\] for every $n\geq 1$.    
	\end{prop}
	
	\begin{proof}
	We fix a set of generators $a_1,\dots,a_r$ of $A$. For every $d\geq 1$, we define $X_d $ and $Y_d$ as in \Cref{essdim}.
		
	Since $r_A(n)\geq r_{A_{\cl{k}}}(n)$, to prove the lower bound for $r_A(n)$ we may assume that $k$ is algebraically closed. Since $A$ is tame, by definition there exists a $\Lambda$-representation $N$ for some $k$-algebra $k[x]\c\Lambda\c k(x)$ parametrizing an infinite number of non-isomorphic indecomposable representations of some dimension $d$. We may view $N$ as the datum of $r$ square matrices of size $d$ and whose coefficients are rational functions of $x$. For all but finitely many $\lambda\in k$ we may specialize $x$ to $\lambda$ in these matrices, obtaining a matrix description of a $d$-dimensional indecomposable representation of $A$. In other words, $N$ defines a rational map $\A^1\dashrightarrow Y_d$.
	
	Similarly, considering $m\geq 1$ copies of $N$ gives a rational map \[\A^m_k\dashrightarrow Y_{md}\c X_{md}=(X_d)^{\oplus m}\] that is ${S}_m$-equivariant. Here $S_m$ acts by permuting the coordinates on the left, and permutes the factors $(X_d)^{\oplus m}$ on the right $S_m$. By the Krull-Schmidt Theorem, at most finitely many $S_m$-orbits map to the same $\on{GL}_{md}$-orbit. We have $r_A(md)=\ed_k\mc{R}_A[md]\geq \ed_k[Y_{md}/\on{GL}_{md}]$. Using \Cref{rosen} with $G=\on{GL}_{md}$ and $H=\on{S}_m$, we deduce the lower bound $r_A(md)\geq m$ for each $m\geq 0$. Since $r_A(n)$ is non-decreasing, the proof of the lower bound for $r_A(n)$ is complete.
	
	We now turn to the proof of the upper bound (so $k$ is again arbitrary). Let $M$ be an indecomposable $A_K$-module, where $K$ is an algebraically closed field containing $k$. Then $M$ is defined over $\cl{k(\lambda)}$ for some $\lambda\in\cl{K}$, by \cite[Lemma 4.6(a)]{kasjanperiodicity}.\footnote{In \cite[Lemma 4.6]{kasjanperiodicity} $A$ is supposed to be finite-dimensional over $k$, however this is not needed for the proof of \cite[Lemma 4.6(a)]{kasjanperiodicity}. The structure constants of $A$ form a countable set, but since $A$ is finitely generated only finitely many intervene in the logical expression which is used in the proof of \cite[Lemma 4.6(a)]{kasjanperiodicity}. This implies that the expression remains a first order formula in this more general
setting.} Since $k(M)=k(M_{\cl{K}})$, we obtain $\trdeg_kk(M)=\trdeg_kk(M_{\cl{K}})\leq \trdeg_k\cl{k(\lambda)}\leq 1$. We may now apply \Cref{algclosedreduction}, to obtain $r_A(n)\leq 2n-1$ for every $n\geq 0$. 
	\end{proof}
	
\begin{rmk}
	The rational maps appearing in the proof of \Cref{trichotomy2} (and also \Cref{trichotomy3} and \Cref{tameed} below), have already been constructed and used by de la Pe{\~n}a \cite[\S 1.4, 1.5]{delapena1991dimension}.
\end{rmk}	
	
	\begin{prop}\label{trichotomy3} 
	Let $k$ be an arbitrary field, and let $A$ be a $k$-algebra. Assume that $A_{\cl{k}}$ is wild. Then there exists a constant $c>0$ such that \[r_A(n)\geq cn^2-1\] for every $n\geq 1$.
	\end{prop}

	\begin{proof}
	Let $a_1,\dots,a_r$ be a set of generators of $A$ and, define $Y_d$ as in \Cref{essdim}, for every $d\geq 1$.
	
    We have $r_A(n)\geq r_{A_{\cl{k}}}(n)$, hence we may assume that $k$ is algebraically closed. Since $A$ is wild, there exists a strict $k\set{x,y}$-representation $N$ of $A$, where $k\set{x,y}$ is the free $k$-algebra on two generators.
    By definition, we have an isomorphism $N\cong k\set{x,y}^{\oplus d}$ of right $k\set{x,y}$-modules, for some $d\geq 1$. Note that the positive integer $d$ is uniquely determined, because it coincides with the dimension of the $k$-vector space $N\otimes_{k\set{x,y}} (k\set{x,y}/(x,y))$.
    
    For every $n\geq 1$, $n$-dimensional representations of $k\set{x,y}$ are in bijective correspondence with arbitrary pairs of square matrices of size $n$ (no relations are enforced). Therefore, we may view $N$ as the datum of $r$ square matrices $P_i(x,y)$ of size $d$ and whose coefficients belong to $k\set{x,y}$. If a $k\set{x,y}$-module $M$ of dimension $m$ corresponds to a pair of matrices $(Q_x,Q_y)$, then $N\otimes_{k\set{x,y}}M$ corresponds to the matrices $P_i(Q_x,Q_y)$. In other words, with the notations of \Cref{essdim}, the association $M\mapsto N\otimes_{k\set{x,y}}M$ gives a rational map
		\[f:\on{M}^{\oplus 2}_{m\times m,k}\dashrightarrow Y_{md}\c X_{md}.\] If $(Q'_x,Q'_y)$ and $(Q_x,Q_y)$ can be obtained one from the other by a simultaneous conjugation, their image in $Y_{md}$ are also related by simultaneous conjugation. This means that the map $f$ is $\on{GL}_{m}$-equivariant, where $\on{GL}_m$ acts by simultaneous conjugation on the left and via the diagonal inclusion $\on{GL}_m\c\on{GL}_{md}$ on the right. 
		
	    The assumption that $N$ is strict implies that $f$ maps distinct $\on{GL}_m$-orbits in $\on{M}^{\oplus 2}_{m\times m}$ to distinct $\on{GL}_{md}$-orbits in $Z_{md}$. We may apply \Cref{rosen} with $H=\on{GL}_m$, $G=\on{GL}_{md}$, $Y=\on{M}^{\oplus 2}_{m\times m}$ and $X=Z_{md}$, and we obtain that $\ed_k[Z_{md}/\on{GL}_{md}]\geq 1+m^2$. Since $r_A(md)\geq\ed_k[Z_{md}/\on{GL}_{md}]$, this shows that that $r_A(md)\geq 1+m^2$ for every $m\geq 1$. Since $r_A(m)$ is non-decreasing, the conclusion follows.
	\end{proof}
	
	We conclude this section with two remarks.
	
	\begin{rmk}\label{groupalgebras}
		Assume that $A=kG$ is the group algebra of a finite group $G$. In this case, representations of $kG$ correspond to representations of $G$. The question of the essential dimension of representations of $G$ was studied in \cite{karpenko2014numerical} and \cite{bensonreichstein}. In the non-modular case $r_{kG}(n)\leq |G|/4$ for every $n\geq 0$, as proved in \cite[Proposition 9.2]{karpenko2014numerical} and \cite[Remark 6.5]{karpenko2014numerical}. On the other hand, if $\on{char}k=p>0$ and $G$ contains a subgroup isomorphic to $(\Z/p\Z)^2$, then \cite[Theorem 14.1]{karpenko2014numerical} shows that $r_{kG}(n)$ becomes arbitrarily large. In \cite[Theorem A.5]{karpenko2014numerical}, it is found that $r_{kG}(n)$ grows at least linearly in $n$.
		
		To see how the results of this paper relate to \cite{karpenko2014numerical} and \cite{bensonreichstein}, recall that there is a complete classification of the representation type of finite group algebras. In characteristic zero, $kG$ is of finite representation type for any finite group $G$. If $\on{char}k=p$ is positive, it is a classical theorem of D. Higman that $\cl{k}G$ is of finite representation type if and only if a Sylow $p$-subgroup of $G$ is cyclic (see \cite[Theorem 2, Theorem 4]{higman1954indecomposable}). Tame group algebras only occur in characteristic $2$, and have been classified by Bondarenko and Drozd \cite{bondarenko1977representation}; every other group algebra is wild.
		
		In view of this classification, \Cref{trichotomy-finite} (and more generally the results of this section) gives a common framework for all the previous results on essential dimension of group algebras. Moreover, it strengthens the lower bound of \cite[Theorem A.5]{karpenko2014numerical} for wild group algebras in characteristic $p$ (i.e. the majority of them) and  for $n$ large enough.
	\end{rmk}

	\begin{rmk}\label{variants}
	We now discuss the assumptions of the previous propositions in the context of representation type over arbitrary fields. Let $k$ be an arbitrary field, and let $A$ be a finite-dimensional $k$-algebra (this restriction is necessary because most of the theorems that we will quote are not known when $A$ is only assumed to be finitely generated). 
	
	If $A_{\cl{k}}$ is of finite representation type, so is $A$; see \cite[Lemma 3.2]{jensen1982homological}. The converse is not true, in general. For example, let $k$ be a non-perfect field of characteristic $p>0$, pick $u\in k^{\times}\setminus k^{\times p}$, and set $K:=k[y]/(y^p-u)$ and $A:=K[x]/(x^p)$. Then $A$ is of finite representation type (see \Cref{nilpotent} below), while $A_{\cl{k}}$ is not; see \cite[Remark 3.4]{jensen1982homological}. On the other hand, we always have $r_A(n)\geq r_{A_{\cl{k}}}(n)$, and so $r_A(n)$ is not bounded from above. It follows that \Cref{trichotomy1} does not hold without the assumption that $k$ is perfect. Note that if $k$ is perfect, then $A$ is of finite representation type if and only if $A_{\cl{k}}$ is of finite representation type; see \cite[Theorem 3.3]{jensen1982homological}. 
	
	There is a notion of \emph{generically tame} $k$-algebra $A$, due to Crawley-Boevey \cite{crawley1991tame}, which generalizes the notion of tameness over an arbitrary field $k$; see also \cite[p. 646]{skowronski2008trends}. If $k$ is perfect and $A$ is generically tame, then $A_{\cl{k}}$ is generically tame; see \cite{kasjan2001base} when $k$ is infinite, and \cite{mendez2013remark} when $k$ is finite. Moreover, by a theorem of Crawley-Boevey \cite{crawley1991tame} (see also \cite[Theorem 1.13(4)]{skowronski2008trends}) the algebra $A_{\cl{k}}$ is generically tame if and only if it is tame. Therefore, the assumptions of \Cref{trichotomy2} hold when $k$ is perfect and $A$ is generically tame. 
	
	Finally, there is a notion of \emph{semi-wild} algebra over an arbitrary field; see \cite[p. 247]{drozd}. If $A$ is a semi-wild $k$-algebra, by \cite[Proposition 2]{drozd} $A_{\cl{k}}$ is also semi-wild. Applying \cite[Theorem 1, Corollary 1]{drozd}, we deduce that $A_{\cl{k}}$ is actually wild. Thus the assumptions of \Cref{trichotomy3} hold when $k$ is arbitrary and $A$ is semi-wild.
\end{rmk}

\section{Algebras admitting a one-dimensional representation}
Let $A$ be a finitely generated $k$-algebra. The purpose of this section is to prove \Cref{nondecreasing1}, which shows that, in some circumstances, in the definition of $r_A(n)$ it is enough to consider modules of dimension exactly $n$, as opposed to $\leq n$. The results of this section hold for group algebras and quiver algebras, hence this reconciles our notation with that of \cite{karpenko2014numerical} and \cite{bensonreichstein}.

\begin{lemma}\label{deffield1}
	Let $K$ be a field containing $k$, $F=K(t_1,\dots, t_r)$ be a purely transcendental field extension of transcendence degree $r$. If $M$ an indecomposable $A_K$-module, then $M_F$ is indecomposable.
\end{lemma}

\begin{proof}	
	It is enough to consider the case $F=K(t)$. Recall that a module is indecomposable if and only if its endomorphism algebra is local. Since $M$ is indecomposable, the algebra $\on{End}_K(M)$ is local, hence the quotient $D:=\on{End}_K(M)/j(\on{End}_K(M))$ is a division algebra over $K$. Denote by $L$ the center of $D$: it is a finite field extension of $K$. Since $L(t)/L$ is purely transcendental, $D\otimes_LL(t)$ is a central division algebra over $L(t)$ (this can be seen directly, or by appealing to \cite[Theorem 1.3]{schofield1992index}). Therefore, $D\otimes_KK(t)\cong D\otimes_LL\otimes_KK(t)\cong D\otimes_LL(t)$ is a division algebra over $K(t)$. We have an inclusion $j(\on{End}_K(M))\otimes_KF\c j(\on{End}_F(M_F))$, hence $\on{End}_F(M_F)/j(\on{End}_F(M_F))$ is a non-zero quotient of $D\otimes_KF \cong D\otimes_LL\otimes_KF \cong D\otimes_LL(t)$. It follows that $\on{End}_F(M_F)/j(\on{End}_F(M_F))\cong D\otimes_KF$ is a division algebra, hence $\on{End}_F(M_F)$ is local and $M_F$ is indecomposable.
\end{proof}

\begin{lemma}\label{summanddefined1}
	Let $L/K$ be an extension of fields containing $k$ and let $M$ be an indecomposable $A_K$-module. If one of the indecomposable summands of $M_L$ is defined over $K$, then $M_L$ is indecomposable. 
\end{lemma}
If $A$ is finite-dimensional, this follows from \cite[Lemma 2.5]{kasjan2000auslander}.
\begin{proof}
	We may assume that $L/K$ is finitely generated. Let $F\c L$ be a purely transcendental extension of $K$ such that $L/K$ has finite degree. By \Cref{deffield1}, $M_F$ is indecomposable, thus we may assume that $L/K$ has finite degree $d$. By assumption, there exists an $A_K$-module $M'$ such that $M'_L$ is an indecomposable summand of $M_L$. There are isomorphisms $M_L\cong M^{\oplus d}$ and $M'_L\cong (M')^{\oplus d}$ of $A_K$-modules. This implies that $M'$ is a direct summand of $M$. Since $M$ is indecomposable, by the Krull-Schmidt Theorem we see that $M'=M$, and so $M_L=M'_L$ is indecomposable.
\end{proof}

\begin{lemma}\label{summanddefined2}
	Let $L/K$ be an extensions of fields containing $k$, and let $M$ be an $A_L$-module. Assume that $M=M'\oplus M''$, where $M'$ and $M''$ are also $A_L$-modules, and suppose that $M$ and every indecomposable summand of $M''$ are defined over $K$. Then $M'$ is defined over $K$ as well.
\end{lemma}

\begin{proof}
	We may write $M=N_L$ and $M''=N''_L$ for some $A_K$-modules $N$ and $N''$. Let $N=\oplus N_i$ be the decomposition of $N$ in indecomposable summands. We have: \[M'\oplus M''\cong M\cong \oplus_i (N_i)_L.\] For fixed $i$, if $(N_i)_L$ shares a direct summand with $M''$, then by the assumptions and by \Cref{summanddefined1} we see that it is an indecomposable summand of $M''$. Therefore, each $(N_i)_L$ is a direct summand of $M'$, $M''$, or both. Let $N'$ be the direct sum of those $N_i$ such that $(N_i)_L$ is a summand of $M'$, and let $N''$ be the direct sum of those $N_i$ such that $(N_i)_L$ is a summand of $M''$ but not of $M'$. Then $N=N'\oplus N''$ and $N'_L=M'$.	
\end{proof}

\begin{lemma}\label{summanddefined3}
	Let $L/k$ be a field extension and let $M$ be an $A_L$-module. Write $M=M'\oplus M''$ and assume that every indecomposale summand of $M''$ is defined over $k$. Then \[\ed_kM=\ed_kM'.\]
\end{lemma}
\begin{proof}
	It is clear that $\ed_kM'\leq\ed_kM$. Let $K/k$ be a field of definition for $M$ of minimal transcendence degree. By \Cref{summanddefined2}, $M'$ is also defined over $K$. It follows that $\ed_kM=\trdeg_kK\geq \ed_kM'$, hence $\ed_kM=\ed_kM'$, as desired.	
\end{proof}

\begin{prop}\label{nondecreasing1}
	Let $A$ be a finitely generated $k$-algebra, and assume that there exists a one-dimensional $A$-module over $k$. Then there exist an extension $K/k$ and an $n$-dimensional $A_K$-module $M$ such that $\ed_kM=r_A(n)$.
\end{prop}

\begin{proof}
	Let $L/k$ be a field extension, and let $M'$ be a $d$-dimensional $A_L$-module such that $d\leq n$ and $\ed_kM'=r_A(n)$, and let $M_0$ be a one-dimensional representation of $A$. By \Cref{summanddefined3}, if $M:=M'\oplus (M_0)^{n-d}_L$, then $M$ is $n$-dimensional and $\ed_kM=\ed_kM'$. This shows that the value $r_A(n)$ is attained among $n$-dimensional modules.
\end{proof}

	\section{Preliminaries on representations of quivers}\label{prelimquiver}
The purpose of this section is to recall the definitions and results of the theory of quiver representations that are relevant to our discussion. We refer the reader to \cite[Chapters 1,2,3]{kirillov} and \cite{schiffler} for detailed accounts of the general theory,  and to \cite[Chapter 7]{kirillov} and \cite[Chapter 14]{simson} for discussions of representation type of quivers.

If $Q$ is a quiver, we denote by $Q_0$ the finite set of its vertices, and by $Q_1$ the finite set of arrows between them.
For every field $K$ there is an equivalence of categories between $K$-representations of $Q$ and modules over the path algebra $KQ$ of $Q$, that is natural with respect to field extensions $L/K$; see \cite[Theorem 5.4]{schiffler}. 

We denote by $\ang{\cdot,\cdot}$ the Tits form of $Q$. By definition, for every $\alpha,\beta\in \mathbb{R}^{Q_0}$ we have:
\[\ang{\alpha,\beta}=\sum_{i\in Q_0}\alpha_i\beta_i-\sum_{a:i\to j}\alpha_i\beta_j,\] where the second sum is over all arrows of $Q$. We denote by $(\cdot,\cdot)$ the associated symmetric bilinear form: $(\alpha,\beta):=\ang{\alpha,\beta}+\ang{\beta,\alpha}$ for every $\alpha,\beta\in\mathbb{R}^{Q_0}$. We let $q_Q$ be the quadratic form defined by $q_Q(\alpha)=\ang{\alpha,\alpha}$ for every $\alpha\in\mathbb{R}^{Q_0}$. We refer the reader to \cite[\S 1.5, \S 1.7]{kirillov} for the definitions of the Cartan matrix $C_Q$, the Weyl group, the simple reflections $s_i$, and the root system of $Q$. We denote by $e_i$ the canonical basis of the vector space $\R^{\Q_0}$. The \emph{fundamental region} is the set $F_Q$ of non-zero $\alpha\in\N^{Q_0}$ with connected support and $(\alpha,e_i)\leq 0$ for all $i$. An imaginary root $\alpha$ is called \emph{isotropic} if $\ang{\alpha,\alpha}=0$ and \emph{anisotropic} if $\ang{\alpha,\alpha}<0$. The \emph{dimension vector} of the representation $M$ is the vector $(\dim M_i)_{i\in Q_0}$.

The quiver $Q$ is said to be of \emph{finite representation type}, \emph{tame} or \emph{wild} if the path algebra $\cl{k}Q$ is of finite representation type, tame or wild, respectively. The representation type of $Q$ does not depend on the field $k$; see below. A quiver is \emph{connected} if its underlying graph is connected. By Gabriel's Theorem \cite{gabriel}, the connected quivers of finite representation type are exactly those whose underlying graph is a Dynkin diagram of type $A$, $D$ or $E$ (see \cite[Theorem 3.3]{kirillov} or \cite[Theorem 8.12]{schiffler}). A connected quiver $Q$ is tame if and only if its underlying graph is an extended Dynkin diagram of type $\widetilde{A}$, $\widetilde{D}$ or $\widetilde{E}$, and it is wild otherwise; see \cite[Theorem 7.47]{kirillov}. 

If $Q$ is connected, the representation type of $Q$ is determined by $C_Q$: $Q$ is of finite representation type if and only if $C_Q$ is positive definite, tame if and only if $C_Q$ is positive semidefinite but not definite, and wild if and only if $C_Q$ is non-degenerate and indefinite; see \cite[\S 8.2]{schiffler} or \cite[Theorem 1.28]{kirillov}. If $Q$ is tame, there is a unique $\delta\in\Z^{Q_0}$ such that $\ang{\delta,\delta}=0$, $\delta_i\geq 1$ for every $i\in Q_0$ and $\min\delta_i=1$, called the \emph{null root} of $Q$. A root $\alpha$ is a \emph{Schur root} if there exist a field extension $K/k$ and a $K$-representation $M$ of $Q$ of dimension vector $\alpha$ such that $\on{End}(M)=K$ (such a representation $M$ is usually called a \emph{brick}). 

The first result related to fields of definitions of quiver representations that we are aware of is the following, due to G. Kac and A. Schofield. 

\begin{prop}\label{deffieldrealroot}
	Let $\alpha$ be a real root for the quiver $Q$. If $K$ is an algebraically closed field, the unique indecomposable representation of dimension vector $\alpha$ is defined over the prime field of $K$.
\end{prop}

\begin{proof}
	See \cite[Theorem 1]{kac1} for the case of positive characteristic, and \cite[Theorem 8]{deffieldrealroot} in characteristic zero.	
\end{proof}
	
	Let $Q$ be a quiver, and let $\alpha$ be a dimension vector for $Q$. We define the functor\[\on{Rep}_{Q,\alpha}:\on{Fields}_k\to\on{Sets}\] by setting
		\[\on{Rep}_{Q,\alpha}(K):=\set{\text{Isomorphism classes of $\alpha$-dimensional $K$-representations of $Q$}}.\] If $K\c L$ is a field extension, the corresponding map $\on{Rep}_{Q,\alpha}(K)\to\on{Rep}_{Q,\alpha}(L)$ is given by tensor product. 

We denote $r_{kQ}(n)$ simply by $r_Q(n)$. Since representations of a quiver $Q$ are the same as representations of its path algebra, for any $n\geq 0$ we have \[r_Q(n)=\max_{\sum\alpha_i\leq  n}\ed_k\on{Rep}_{Q,\alpha}.\] By \Cref{nondecreasing2} below, one may equivalently to take the maximum over those $\alpha$ which satisfy $\sum\alpha_i=n$.	

	\begin{lemma}\label{nondecreasing2}
	    Let $Q$ be a quiver. If $\alpha,\beta$ are two dimension vectors such that $\beta_i\leq\alpha_i$ for each vertex $i$ of $Q$, then \[\ed_kR_{Q,\beta}\leq \ed_kR_{Q,\alpha}.\]
	\end{lemma} 
	\begin{proof}
	 	This follows from an application of \Cref{summanddefined3} as in the proof of \Cref{nondecreasing1}, this time by letting $M''$ be the trivial representation of $Q$ of dimension vector $\alpha-\beta$.   
	\end{proof}

\begin{rmk}
	We record here another interesting consequence of \Cref{summanddefined3}. We will not use it in the sequel. Recall that if $Q$ is a quiver without oriented cycles, the category of its finite-dimensional representations has enough projectives (see \cite[Theorem 1.19]{kirillov}), and so we may consider its stable category. Since $Q$ has no oriented cycles, the projective representations of $Q$ are finite-dimensional and are defined over the base field $k$; see \cite[Theorem 1.18]{kirillov}. 
	
	Let $M,N$ be representations of $Q$. If $M$ and $N$ are stably equivalent $kQ$-modules, there is an isomorphism $M\oplus P_1\cong N\oplus P_2$, where $P_1$ and $P_2$ are projective representations for $Q$. By \Cref{summanddefined3} we have \[\ed_kM=\ed_k(M\oplus P_1)=\ed_k(N\oplus P_2)=\ed_kN.\] It follows that essential dimension is a \emph{stable invariant}.
	
	Note that an analogous assertion fails in the setting of central simple algebras and is an open problem in the case of quadratic forms; see \cite[\S 7.4]{reichstein2010essential}.
	\end{rmk}

\section{Stacks of quiver representations}
	Analogously to what we have done for algebras in \Cref{essdim}, we may view $K$-representations of a quiver $Q$ as $K$-orbits of a suitable action. Let $X_{Q,\alpha}:=\prod_{a:i\to j}\on{M}_{\alpha_j\times\alpha_i,k}$ and let $G_{Q,\alpha}:=\prod_{i\in Q_0} \on{GL}_{\alpha_i}$ be an affine space and an algebraic group over $k$, respectively. There is an action of $G_{Q,\alpha}$ over $X_{Q,\alpha}$, given by \[(g_i)_{i\in Q_0}\cdot (P_{a})_{a:i\to j}:=(g_jP_{a}g_i^{-1})_{a:i\to j}.\] We denote by $\mc{R}_{Q,\alpha}$ the quotient stack $[X_{Q,\alpha}/G_{Q,\alpha}]$. As for algebras, one can show that there is a bijection between $K$-points of $\mc{R}_{Q,\alpha}$ and isomorphism classes of representations of $Q$ of dimension $\alpha$. The stack $\mc{R}_{Q,\alpha}$ is a smooth stack of finite type over $k$. 

If $S$ is a $k$-scheme, an $S$-representation of $Q$ is given by a locally free $\O_S$-module $E_i$ for each vertex $i$ of $Q$, and by an $\O_S$-linear homomorphism $\phi_a:E_i\to E_j$ for each arrow $a:i\to j$. For any natural number $n$, let $\mc{N}il^n_Q$ denote the stack over $\on{Sch}_k$ parametrizing representations $M$ of $Q$ over $S$, together with a morphism $\psi:M\to M$ such that $\psi^n=0$, and such that $\coker\psi^j$ is an $S$-representation for every $j\geq 0$ (i.e. for each vertex the corresponding coherent sheaf is locally free). We note that $\mc{N}il^0_Q=\Spec k$ and $\mc{N}il^1_Q$ is the disjoint union of the $\mc{R}_{Q,\alpha}$, where $\alpha$ ranges among over all possible dimension vectors. 

We are going to follow \cite[\S 6]{biswasdhillonhoffmann} for various statements and proofs in this section. In \cite[\S 6]{biswasdhillonhoffmann}, the authors studied stacks of coherent sheaves on a fixed curve $C$, and analyzed them in terms of certain stacks $\mc{N}il_{C,n}$. Our definition of $\mc{N}il^n_Q$ is the analogue of their $\mc{N}il_{C,n}$ in the context of quiver representations. As we state below, the results of \cite[\S 6]{biswasdhillonhoffmann} still hold in this setting. The common feature of the two set-ups that makes the arguments of \cite[\S 6]{biswasdhillonhoffmann} possible is the vanishing of the $\on{Ext}^i$ for $i\geq 2$ (for quiver representations, this follows from \cite[Theorem 2.24]{schiffler}).

\begin{prop}\label{nilpdim} The stack $\mc{N}il^n_Q$ is smooth over $k$, and its dimension at a point $(M,\psi)$ is given by the formula
	\[\on{dim}_{(M,\psi)}\mc{N}il^n_Q=-\sum_{h=1}^n\ang{\beta_h,\beta_h}\]
	where $\beta_h$ is the dimension vector of $\on{im}\psi^{h-1}/\on{im}\psi^h$.
\end{prop}

\begin{proof}
	The proof of smoothness of $\mc{N}il^n_Q$ proceeds as in \cite[Corollary 6.2]{biswasdhillonhoffmann}, with one modification. One step of the proof of \cite[Corollary 6.2]{biswasdhillonhoffmann} rests on \cite[Lemma 3.8]{biswas2008some}. The analogous result for quiver representations is still true, and is a direct application of the infinitesimal criterion for smoothness. The computation of the dimension of $\mc{N}il^n_Q$ also closely follows the argument of \cite[Corollary 6.2]{biswasdhillonhoffmann}, using the fact that if $N$ is a $K$-representation of $Q$ of dimension vector $\alpha$, then
	\[\ang{\alpha,\alpha}=\on{dim}_K(\on{End}(N))-\on{dim}_K(\on{Ext}^1(N,N));\] see \cite[Proposition 8.4]{schiffler}.
\end{proof}

In the sequel, we will only use the following corollary of \Cref{nilpdim}.

\begin{cor}\label{trdegindec}
	Let $M$ be an indecomposable representation of $Q$ over an algebraically closed field $K$ containing $k$, and let $\alpha$ be its dimension vector. If $\beta_j$ denotes the dimension vector of $\on{im}(\psi^{j-1})/\on{im}(\psi^j)$ for a general element $\psi$ of the Jacobson radical $j(\on{End}(M))$, then
	\[\trdeg_kk(M)\leq 1-\sum_j\ang{\beta_j,\beta_j}.\]
\end{cor}

\begin{proof}
	The result follows from \Cref{nilpdim}, in the same way that \cite[Corollary 6.3]{biswasdhillonhoffmann} follows from \cite[Corollary 6.2]{biswasdhillonhoffmann}.
\end{proof}

\section{Quivers of finite and tame representation type}\label{nonwild}
The remaining part of this article is concerned with the essential dimension of quiver representations. We begin by considering quivers of finite and tame representation type. 

\begin{prop}\label{finrep}
    Let $k$ be an arbitrary field, and let $Q$ be a quiver.
    \begin{enumerate}[label=(\alph*)]
        \item\label{finrep1} If $\alpha$ is a real root for $Q$, then $\ed_k\on{Rep}_{Q,\alpha}=0$.
        \item\label{finrep2} If $Q$ is of finite representation type, then $r_Q(n)=0$ for every $n\geq 1$.
    \end{enumerate}
\end{prop}

\begin{proof}
  (a) If $K/k$ is a field extension, and $M$ is an $\alpha$-dimensional $K$-representation of $Q$, every indecomposable summand of $M_{\cl{K}}$ is defined over the prime field of $k$; see \cite[Theorem 1]{kac1} for the case of positive characteristic, and \cite[Theorem 8]{deffieldrealroot} in characteristic zero. By Noether-Deuring's Theorem, $M$ is defined over the prime field of $k$, hence $\ed_kM=0$. This implies $\ed_k\on{Rep}_{Q,\alpha}=0$ for every dimension vector $\alpha$, and $r_Q(n)=0$ for every $n\geq 1$. 
    
  (b) By Gabriel's Theorem (as stated in \cite[Theorem 8.12]{schiffler}), the dimension vector of an indecomposable representation of $Q$ is a positive real root. Now (b) follows from (a).
\end{proof}

If $Q$ is tame, the computation of $r_Q(n)$ will follow from \Cref{nondecreasing2} and the following lemma, proved in the Appendix.

\begin{lemma}\label{uppertame}
	Let $K$ be a field extension of $k$, and $M$ an indecomposable representation of a tame quiver $Q$ of dimension vector $m\delta$ over $K$. Then there exist $a_1,\dots,a_m\in K$, and bases of the vector spaces $M_i$, for $i\in Q_0$, so that the linear maps $\phi_a$, $a\in Q_1$, are represented by matrices having entries in $\set{0,1,a_1,\dots,a_m}$.
\end{lemma}

In \cite{donovanfreislich}, the indecomposable representations of a tame quiver $Q$ are classified, over an algebraically closed field $K$. Another reference on this topic is \cite[Chapter XIII]{simson2007elements}. Each indecomposable representation may be described by matrices having entries in $\set{0,1,\lambda}$, for some $\lambda\in K$. In the Appendix we show that this classification may be naturally extended to arbitrary fields, with the help of some results successive to \cite{donovanfreislich}, namely \cite{dlab1976indecomposable} and \cite{medina2004four}. This is analogous to the passage from the Jordan canonical form to the rational canonical form;  see \Cref{basicexample}.

\begin{prop}\label{tameed}
	Let $Q$ be a tame connected quiver, with null root $\delta$. If $\alpha$ is a dimension vector, and $m$ is the biggest non-negative integer such that $m\delta_i\leq\alpha_i$ for each vertex $i$ of $Q$, then $\on{ed}_k\on{Rep}_{Q,\alpha}=m$.
	Furthermore, \[r_{Q}(n)=\floor{\frac{n}{\sum\delta_i}}.\] 
	\end{prop}
In particular, $a_1(kQ)={1}/{\sum\delta_i}$.
\begin{proof}
	Let $K$ be a field containing $k$, $\alpha$ a dimension vector, and $M$ an $\alpha$-dimensional $K$-representation.  Then $M$ decomposes as a direct sum of indecomposable representations $M_h$, and \[\ed_kM\leq \sum\ed_kM_h.\] Let $m$ be the maximum non-negative integer such that $m\delta_i\leq \alpha_i$ for each vertex $i$ of $Q$.
	If the dimension vector of $M_h$ is a real root, then $\ed_kM_h=0$ by \Cref{finrep}(a). If it is $m_h\delta$, then by \Cref{uppertame} we have $\ed_k M_h\leq m_h$. Since $\sum m_h\leq m$, we conclude that $\ed_kM\leq m$. Therefore $\ed_k\on{Rep}_{Q,\alpha}\leq m$.
	
	Let us now prove that $\ed_k\on{Rep}_{Q,\alpha}\geq m$. Since $\ed_k\on{Rep}_{Q,\alpha}\geq \ed_{\cl{k}}\on{Rep}_{Q,\alpha}$, we may assume that $k$ is algebraically closed. Let $Z_m\c X_{Q,m\delta}$ be the locally closed subset parametrizing representations $\oplus_{h=1}^m M_h$, where each $M_h$ has dimension vector $\delta$. There is an obvious action of $S_m$ on $Z_m$, given by permutation of the summands. Consider $m$ copies of an infinite family of indecomposable representations of dimension vector $\delta$ parametrized by an open subset of $\A^1$. As in the proof of \Cref{trichotomy2}, this gives an $S_m$-equivariant rational map \[\A^m\dashrightarrow Z_m\] sending at most finitely many $S_m$-orbits of $\A^m$ to the same $G_{\alpha}$-orbit of $Z_m$. By \Cref{rosen}, we deduce that $\ed_k\on{Rep}_{Q,m\delta}\geq m$. By \Cref{nondecreasing2}, we have that $\ed_k\on{Rep}_{Q,\alpha}\geq m$, hence we conclude that $\ed_k\on{Rep}_{Q,\alpha}=m$ as desired.
	
	We now prove the formula for $r_Q(n)$. Set \[d:=\floor{\frac{n}{\sum\delta_i}}.\] Fix a non-negative integer $n$, and let $\alpha$ be a dimension vector such that $\sum\alpha_i\leq n$. Let $m$ the maximum non-negative integer for which the inequality $m\delta_i\leq \alpha_i$ holds for each vertex $i$ of $Q$. By what we have proved so far, $\ed_k\on{Rep}_{Q,\alpha}=m$. By summing all the inequalities $m\delta_i\leq \alpha_i$ we obtain $m\sum\delta_i\leq \sum\alpha_i\leq n$, so that $m\leq d$. This implies 
	\[r_{Q}(n)=\max_{\sum\alpha_i\leq n}\ed_k\on{Rep}_{Q,\alpha}\leq d.\]
	On the other hand, we may choose $\alpha$ such that $d\delta_i\leq\alpha_i$ for each vertex $i$. In this case $\ed_k\on{Rep}_{Q,\alpha}=d$, and the proof of the formula for $r_Q(n)$ is complete.	
\end{proof}

\section{Wild quivers}\label{wilda2}

In this section we determine $a_2(kQ)$ for every wild quiver $Q$. Let $M$ be a $K$-representation of $Q$, for some field $K$ containing $k$. Recall that by \Cref{edkm} the term $\ed_{k(M)}M$ grows sublinearly with the dimension of $M$, so the quadratic contribution to $r_Q(n)$ will come from  $\trdeg_kk(M)$. Our first objective is to produce lower and upper bounds for the term $\trdeg_kk(M)$.
\begin{lemma}\label{schurlower}
	Let $\alpha$ be a Schur root for $Q$. Then there exists an $\alpha$-dimensional representation $M$ of $Q$ such that $\trdeg_kk(M)\geq 1 -\ang{\alpha,\alpha}$.
\end{lemma}

\begin{proof}
	We may assume that $k$ is algebraically closed. Since $\alpha$ is a Schur root, by \cite[Proposition 4.4]{king} there exists a non-empty coarse moduli space $M_{Q,\alpha}^{\on{st}}$ for stable $\alpha$-dimensional representations of $Q$, and it is irreducible of dimension $1-\ang{\alpha,\alpha}$. There is a dominant rational map $\mc{R}_{Q,\alpha}\dashrightarrow M_{Q,\alpha}^{\on{st}}$, that is, a dominant morphism from a non-empty open substack of $\mc{R}_{Q,\alpha}$. Let $\eta$ be the generic point of $M_{Q,\alpha}^{\on{st}}$.
	We have \[\on{trdeg}_kk(\eta)=\on{dim}M_{Q,\alpha}^{\on{st}}=1-\ang{\alpha,\alpha}.\] Let $M$ be a representation over a field $K$ such that the composition $\Spec K\to \mc{R}_{Q,\alpha}\dashrightarrow M_{Q,\alpha}^{\on{st}}$ is well defined and has image equal to $\eta$.  Any field of definition for $M$ must contain $k(\eta)$, hence \[\trdeg_kk(M)\geq\trdeg_kk(\eta)=1-\ang{\alpha,\alpha}.\qedhere\]
\end{proof}
Before proving an upper bound for $\trdeg_kk(M)$, we set some notation.
\begin{defin}
Let $Q$ be a quiver. For a vector $v\in \R^{Q_0}$, we denote by $|v|$ the sum of its coordinates. We define
\begin{align*}
H_Q:=&\set{\alpha\in\R^{Q_0}: |\alpha|=1},\\
S_Q:=&\set{\alpha\in H_Q:\alpha_i\geq 0\ \forall i\in Q_0},\\
\mathring{S}_Q:=&\set{\alpha\in H_Q:\alpha_i> 0\ \forall i\in Q_0}.
\end{align*}
 We denote by $\Lambda_Q$ the maximum of the opposite of the Tits form $-q_Q$ on $S_Q$.
\end{defin}
We note that $\Lambda_Q>0$ if and only if the quiver $Q$ is wild. 
\begin{lemma}\label{loose}	
	Let $K/k$ be a field extension, and let $M$ be an $\alpha$-dimensional $K$-representation of a wild quiver $Q$. Then \[\on{trdeg}_kk(M)\leq |\alpha|+\Lambda_Q|\alpha|^2.\]
\end{lemma}

\begin{proof}
	Assume first that $K$ is algebraically closed and that $M$ is indecomposable.
	By \Cref{trdegindec} there exist dimension vectors $\beta_1,\dots,\beta_h$ such that $\sum\beta_h=\alpha$ and \[\trdeg_kk(M)\leq 1-\sum_h\ang{\beta_h,\beta_h}.\] By definition of $\Lambda_Q$,
	\[-\ang{\beta_h,\beta_h}\leq |\beta_h|^2\Lambda_Q,\] hence
	\[\trdeg_kk(M)\leq 1+\Lambda_Q(\sum_h|\beta_h|^2)\leq 1+\Lambda_Q|\alpha|^2.\]
	Let now $K$ be an arbitrary field extension of $k$, and let $M$ be a representation of $Q$ over $K$ of dimension vector $\alpha$. Since $k(M)=k(M_{\cl{K}})$, we may assume that $K$ is algebraically closed. The representation $M$ decomposes in at most $|\alpha|$ indecomposable representations $M_h$. Let $\alpha_h$ be the dimension vector of $M_h$. Then \[\trdeg_kk(M)\leq\sum_h \trdeg_kk(M_h)\leq \sum_h(1+\Lambda_Q|\alpha_h|^2)\leq |\alpha|+\Lambda_Q|\alpha|^2.\qedhere\]	
\end{proof}

From \Cref{loose} we see that, to understand the asymptotic behavior of $r_Q(n)$, we must first understand $\Lambda_Q$. 

\begin{lemma}\label{critical}
	Assume that $Q$ is a disjoint union of wild connected quivers. There is at most one critical point $\alpha\in \R^{Q_0}$ of $q_Q$ on $H_Q$. If it exists, it satisfies the equations \[(\alpha,e_i)=-2\lambda\] for each vertex $i$ and for some constant $\lambda$. The corresponding critical value of $q_Q$ is \[q_Q(\alpha)=\ang{\alpha,\alpha}=-{\lambda}.\] Moreover, $\lambda\in \Q$ and $\alpha_i\in\Q$ for every $i$.
\end{lemma}

By definition, a critical point of $q_Q$ on $H_Q$ is a point at which all partial derivatives of $q_Q|_{H_Q}$ vanish. It is not necessarily a minimum or maximum of $q_Q|_{H_Q}$; see \Cref{k3k3}.

\begin{proof}
	We use the method of Lagrange multipliers. The constraint is given by \[\alpha_1+\dots+\alpha_n=1,\] and the partial derivatives of $q_Q$ are \[\partial_iq_Q(\alpha)=-(\alpha,e_i).\] Therefore, any critical point $\alpha$ must satisfy the equations \[(\alpha,e_i)=-2\lambda.\] for some $\lambda\in \mathbb{R}$. If $C_Q$ is the Cartan matrix of $Q$, these equations translate to \[C_Q\alpha=-2\lambda e,\] where $e=(1,\dots,1)$. Since $Q$ is a disjoint union of wild connected quivers, $C_Q$ is invertible. Therefore, a critical point $\alpha$ will lie in the intersection of the affine plane $H_Q$ with the line generated by $C_Q^{-1}e$, so there can be at most one. If a critical point $\alpha$ exists, the corresponding critical value of $q_Q$ is \[\ang{\alpha,\alpha}=\frac{1}{2}(\alpha,\alpha)=\frac{1}{2}\sum \alpha_i(\alpha,e_i)=-{\lambda}.\qedhere\]
\end{proof}

\begin{defin}
	Let $Q$ be a disjoint union of connected wild quivers. We denote by $\alpha_Q$ and $\lambda_Q$ the critical point $\alpha$ and the constant $\lambda$ of the previous lemma (if they exist). We say that $Q$ is \emph{effective} if $\alpha_Q$ exists, $\alpha_Q\in \mathring{S}_Q$ and $\lambda_Q>0$.
\end{defin}

Recall that a {subquiver} of $Q$ is a quiver $Q'$ such that $Q'_0\c Q_0$ and whose arrows are all the arrows of $Q$ between vertices in $Q'_0$. If $Q'$ is a subquiver of $Q$, the inclusion $Q'_0\c Q_0$ naturally identifies $\R^{Q'_0}$ with a subspace of $\R^{Q_0}$. 

\begin{lemma}\label{maxwild}
	Let $Q$ be an effective quiver with connected components $Q'_1,\dots, Q'_d$, and assume that $\lambda_{Q'_1}\geq \lambda_{Q'_h}$ for every $h=1,\dots,d$. Then $Q'_1$ is effective, and $\lambda_Q\leq \lambda_{Q'_1}$. 
\end{lemma}

\begin{proof}
	Since $Q$ is effective, $\alpha_Q$ exists and belongs to $\mathring{S}_Q$. We may write $\alpha_Q=(\alpha'_1,\dots,\alpha'_d)$, where $\alpha'_h\in\R^{Q'_h}$ for each $h=1,\dots,d$. Since $q_{Q}$ is the orthogonal direct sum of the $q_{Q'_h}$, we have that \[q_Q(\alpha_Q)=q_{Q_1'}(\alpha'_1)+\dots+q_{Q_d}'(\alpha'_d).\] Note that $\alpha'_h/|\alpha'_h|$ is a critical point for $q_{Q'_h}|_{H_{Q'_h}}$, for every $h=1,\dots,d$. Since the coordinates of $\alpha'_h/|\alpha'_h|$ sum to $1$, from the uniqueness part of \Cref{critical} we obtain that $\alpha'_h/|\alpha'_h|=\alpha_{Q'_h}$, i.e. $\alpha'_h=|\alpha'_h|\alpha_{Q'_h}$. In particular, $\alpha_{Q'_h}$ exists and belongs to $\mathring{S}_{Q'_h}$, for every $h$. Substituting into the previous equation, we obtain \[\lambda_Q=|\alpha'_1|^2\lambda_{Q'_1}+\dots+|\alpha'_d|^2\lambda_{Q'_d}.\]
	Since $Q$ is effective, we have $\lambda_Q>0$. It is thus impossible that $\lambda_{Q'_h}\leq 0$ for every $h$, and so $\lambda_{Q'_1}>0$. We already showed that $\alpha_{Q'_1}\in \mathring{S}_{Q'_1}$, hence $Q'_1$ is effective.
	
	Since $|\alpha'_h|\geq 0$ for every $h$ and $|\alpha'_1|+\dots+|\alpha'_d|=1$, we have $|\alpha'_h|^2\leq |\alpha'_h|$ for all $h$. We conclude that
	\begin{equation*}\lambda_Q\leq (|\alpha'_1|^2+\dots+|\alpha'_d|^2)\lambda_{Q'_1}\leq(|\alpha'_1|+\dots+|\alpha'_d|)\lambda_{Q'_1}=\lambda_{Q'_1},\end{equation*}
	as desired.
\end{proof}

Recall that, by definition, $\Lambda_Q$ is the maximum of $-q_Q$ on $S_Q$. If the maximum of $-q_Q$ is attained in $S_Q$, then by \Cref{critical} we have $\Lambda_Q=\lambda_Q$. If $\alpha$ belongs to the boundary of $S_Q$ (as in \Cref{k3k3} below), we may consider the subquiver $Q'$ whose vertices correspond to the non-zero entries of $\alpha$. If all connected components of $Q'$ were wild, then by \Cref{critical} we would get $\Lambda_Q=\Lambda_{Q'}=\lambda_{Q'}$. In fact, as we now show, one may always arrange for $Q'$ to be wild and connected. We thus obtain the following formula for $\Lambda_Q$, which will be used in the proof of \Cref{wild}.

\begin{prop}\label{maximum}
	Let $Q$ be a wild quiver. Then there exists an effective subquiver of $Q$. Moreover, we have \[\Lambda_Q=\max_{Q'}\lambda_{Q'},\] where the maximum is taken over all effective wild connected subquivers $Q'$ of $Q$.
\end{prop}

\begin{proof}
	Let $Q'$ be an effective wild connected subquiver of $Q$, and view $\alpha_{Q'}$ as a vector in $\R^{Q_0}$ by setting the extra coordinates equal to zero. Then $\Lambda_Q\geq -q_Q(\alpha_{Q'})= -q_{Q'}(\alpha_{Q'})=\lambda_{Q'}$. Letting $Q'$ vary, we obtain $\Lambda_Q\geq\max_{Q'}\lambda_{Q'}$. It thus suffices to find $Q'$ such that $\Lambda_Q=\lambda_{Q'}$.
	
	Since $S_Q$ is compact and $q_Q$ is continuous, there exists a vector $\alpha\in S_Q$ minimizing $q_Q|_{S_Q}$, that is, satisfying $\Lambda_Q=-q_Q(\alpha)$. Since $Q$ is wild, $q_Q$ is indefinite, hence $-q_Q(\alpha)>0$. Let $Q'$ be the subquiver of $Q$ defined by \[Q'_0:=\set{i\in Q_0:\alpha_{i}\neq 0}.\] If we regard $\alpha$ also as a vector in $\R^{Q'_0}$, then $q_Q(\alpha)=q_{Q'}(\alpha)$, $\alpha\in \mathring{S}_{Q'}$, and $\alpha$ minimizes $q_{Q'}$ on ${S}_{Q'}$. In particular, $\alpha$ is a critical point for $q_{Q'}|_{H_{Q'}}$. 
	
	The subquiver $Q'$ is wild, because $q_{Q'}(\alpha)<0$. Now, if $Q''$ is a non-wild connected component of $Q'$, define a vector $\beta\in\R^{Q'_0}$ by setting $\beta_i=0$ if $i\in Q''$ and $\beta_i=\alpha_{i}$ otherwise, and set $\gamma:=\beta/|\beta|$. Then $|\gamma|=1$ and  $\gamma_i\geq 0$ for each vertex $i$ of $Q'$, i.e. $\gamma\in S_Q$. Moreover, since $q_{Q''}$ is positive semi-definite and $q_{Q'}=q_{Q''}\perp q_{Q'\setminus Q''}$, we have $q_{Q'}(\gamma)\leq q_{Q'}(\alpha)$, hence $q_{Q'}(\gamma)= q_{Q'}(\alpha)$. Here $Q'\setminus Q''$ is the subquiver of $Q'$ with set of vertices equal to $Q'_0\setminus Q''_0$. Since $\gamma$ is supported in $Q'\setminus Q''$, this shows that we can remove $Q''$ from $Q'$, that is, we may assume that every connected component of $Q'$ is wild. 
	
	By \Cref{critical}, $\alpha=\alpha_{Q'}$ is the unique critical point of $q_{Q'}|_{H_{Q'}}$ and $\Lambda_Q=-q_Q(\alpha)=-q_{Q'}(\alpha)=\lambda_{Q'}$.  As $Q$ is wild, we have $\Lambda_Q>0$, hence $\lambda_{Q'}>0$. We already noted that $\alpha\in \mathring{S}_Q$, hence $Q'$ is effective. By \Cref{maxwild}, we are allowed to pass to a connected component of $Q'$. Therefore, we may assume that $Q'$ is connected, and this concludes the proof. 
\end{proof}

\begin{example}\label{k3k3}
	We illustrate \Cref{maximum} and its proof by computing $\Lambda_Q$ in the case where $Q$ is the disjoint union of two quivers $K_3$:
	\[
	\begin{tikzcd}[row sep=1ex]
	& 1 \arrow[r] \arrow[r,shift left=.8ex] \arrow[r,shift right=.8ex] & 2 	& 3 \arrow[r] \arrow[r,shift left=.8ex] \arrow[r,shift right=.8ex] & 4. 
	\end{tikzcd}
	\] 
	The Tits form of $Q$ is \[q_Q(\alpha)=\alpha_1^2+\alpha_2^2+\alpha_3^2+\alpha_4^2-3\alpha_1\alpha_2-3\alpha_3\alpha_4.\] As in the proof of \Cref{critical}, we may compute the critical point $\alpha_Q$ of $q_Q|_{H_Q}$ as the solution of the system of linear equations \[3\alpha_2-2\alpha_1=3\alpha_1-2\alpha_2=3\alpha_3-2\alpha_4=3\alpha_4-2\alpha_3,\quad \alpha_1+\alpha_2+\alpha_3+\alpha_4=1.\] Thus $\alpha_Q=(\frac{1}{4},\frac{1}{4},\frac{1}{4},\frac{1}{4})$ and $\lambda_Q=\frac{1}{8}$. In particular, $Q$ is effective. Note that $\lambda_Q$ is not the maximum of $-q_Q|_{H_Q}$, because $-q_Q$ is not bounded from above on $H_Q$: \[-q_Q(t,t,\frac{1}{2}-t,\frac{1}{2}-t)=2t^2+O(t),\quad (t\to\infty).\] 
	Let $Q_1$ and $Q_2$ denote the two $K_3$ subquivers of $Q$. The associated critical points are $\alpha_{Q_1}=(\frac{1}{2},\frac{1}{2},0,0)$ and $\alpha_{Q_2}=(0,0,\frac{1}{2},\frac{1}{2})$, and the corresponding critical value is $\frac{1}{4}$ in both cases. Note that this is accordance with \Cref{maxwild}. In particular, $Q_1$ and $Q_2$ are effective. By \Cref{maximum}, we obtain $\Lambda_Q=\lambda_{Q_1}=\lambda_{Q_2}=\frac{1}{4}$.
	
	For this particular $Q$, the compactness argument in the second paragraph of the proof of \Cref{maximum} would select either $\alpha_{Q_1}$ or $\alpha_{Q_2}$ as $\alpha$, hence either $Q_1$ or $Q_2$ as $Q'$. 
\end{example}

\begin{defin}
    We say that the wild quiver $Q$ is a \emph{minimal wild quiver} if there are no wild subquivers of $Q$ other than $Q$. We say that the underlying graph of $Q$ is a \emph{minimal wild graph} if for every arrow $a:i\to j$ of $Q$, removing $a$ gives a quiver $Q'$ that is not wild.
\end{defin} 
Recall that a subquiver of $Q$ is obtained by choosing a subset of $Q_0$ and considering all arrows between said vertices. Hence subquivers of $Q$ bijectively correspond to subsets of $Q_0$. Note that if the underlying graph of $Q$ is a minimal wild graph, then $Q$ is a minimal wild quiver, but the converse need not be true.

As an example, consider the generalized Kronecker quiver $K_r$:
	\[
\begin{tikzcd}[row sep=1ex]
& 1 \arrow[r] \arrow[r,shift left=.8ex, "r"] \arrow[r,shift right=.8ex] & 2 
\end{tikzcd}
\] 
If $r\geq 3$, $K_r$ is a minimal wild quiver. Its underlying graph is a minimal wild graph if and only if $r=3$. Another example is the loop quiver $L_r$ for $r\geq 2$: it is a minimal wild quiver because it has no non-empty subquivers, but its underlying graph is minimal wild if and only if $r=2$.

There are $18$ minimal graphs, and they are listed in \cite[Lecture 6, Subsection 6.7]{leit6}. They are all obtained by adjoining one vertex to a tame quiver of at most $9$ vertices; see the picture of  \cite[Lecture 6, p. 9]{leit6}.

\begin{lemma}\label{minimaleffective}
	Let $Q$ be a minimal wild quiver. Then $Q$ is connected and effective, and $\Lambda_Q=\lambda_Q$.
\end{lemma}	

\begin{proof}
	Since a disjoint union of non-wild quivers is non-wild, a minimal wild quiver is connected. Moreover, every proper subquiver of $Q$ is non-wild. The claim now follows from \Cref{maximum}.
\end{proof}

Before stating the main result of this section, we need one last lemma.

\begin{lemma}\label{itsschur}
	Let $Q'$ be an effective wild connected subquiver of $Q$, and let $m$ be a positive integer such that $m\alpha_{Q'}$ has integral entries. Then $m\alpha_{Q'}$ is a Schur root of $Q$. 
\end{lemma}

\begin{proof}
Since $Q'$ is effective, we have $\lambda_{Q'}<0$. It follows from \Cref{critical} that \[(m\alpha_{Q'},e_i)=-2m\lambda_{Q'}<0\] for every vertex $i$ of $Q'$. This implies that the support of $m\alpha_{Q'}$ is a wild quiver and that $m\alpha_{Q'}$ belongs to the fundamental region of $Q'$. By \cite[Proposition 4.14]{lebruyn}, it follows that $m\alpha_{Q'}$ is a Schur root of $Q'$. Let $K/k$ be a field extension and let $M'$ be a $K$-representation of $Q'$ such that $\on{End}_K(M')=K$. We may trivially extend $M'$ to a $K$-representation $M$ of $Q$, by letting $M_i=0$ for every $i\in Q_0\setminus Q'_0$. We have $\on{End}_K(M)=\on{End}_K(M')=K$, hence $\alpha_{Q'}$ is a Schur root for $Q$.
\end{proof}

We are now ready to compute $a_2^{\pm}(kQ)$ for every wild quiver $Q$. This completes the proof of \Cref{wild-intro}. 

\begin{prop}\label{wild}
	Let $Q$ be a wild quiver. Then:
	\begin{enumerate}[label=(\alph*)]
		\item\label{wild1} $a_2(kQ)=\Lambda_Q$;
		\item\label{wild2} $a_2(kQ)\geq \frac{1}{2480}$;
		\item\label{wild3} we have $a_2(kQ)=\frac{1}{2480}$ if and only if the underlying diagram of $Q$ is a disjoint union of a (possibly empty) non-wild quiver and copies of the graph $\widetilde{\widetilde{E_8}}$
		\[
		\begin{tikzcd}
		&& 4 \\
		1 \arrow[r,-]
		& 2 \arrow[r,-]
		& 3 \arrow[r,-] \arrow[u,-]
		& 5 \arrow[r,-]
		& 6 \arrow[r,-]
		& 7 \arrow[r,-]
		& 8 \arrow[r,-]
		& 9 \arrow[r,-]
		& 10.
		\end{tikzcd}
		\]\noindent
	\end{enumerate}
\end{prop}

\begin{proof}
(a) Let $n\geq 1$, $K/k$ be a field extension, and $M$ be an $\alpha$-dimensional $K$-representation of $Q$, where $|\alpha|=n$. By \Cref{edkm}, we know that $\ed_{k(M)}M<n$. Furthermore, by \Cref{loose}, $\trdeg_kk(M)\leq {n}+\Lambda_Qn^2$. Thus \[\frac{r_{Q}(n)}{n^2}<\frac{2}{n}+\Lambda_Q.\] Letting $n$ tend to infinity, we obtain \begin{equation}\label{a2+}
a_2^+(kQ)=\limsup_{n\to\infty}\frac{r_{Q}(n)}{n^2}\leq\Lambda_Q.
\end{equation}
	We now establish the opposite inequality. By \Cref{maximum}, there exists an effective wild connected subquiver $Q'$ of $Q$ such that $\lambda_{Q'}=\Lambda_Q$ (and in particular $\lambda_{Q'}>0$, as $Q$ is wild). By \Cref{critical}, $\alpha_{Q'}\in\Q_{>0}^{Q_0}$, hence there exists $m\geq 1$ such that $m\alpha_{Q'}\in \N^{Q_0}$. By \Cref{itsschur}, the vector $m\alpha_{Q'}$ is a Schur root. By \Cref{schurlower}, there exists a representation $M$ of dimension vector $m\alpha_{Q'}$ such that \[\trdeg_kk(M)\geq 1-\ang{m\alpha_{Q'},m\alpha_{Q'}}=1-m^2\ang{\alpha_{Q'},\alpha_{Q'}}=1+\Lambda_Qm^2.\]
	Since $|\alpha_{Q'}|=1$, we have $|m\alpha_{Q'}|=m$. Considering multiples of $m\alpha_{Q'}$ yields \[r_{Q}(mh)\geq 1+\Lambda_Q (mh)^2\] for every non-negative integer $h$. Let now $n$ be a positive integer. There exists a unique $h\geq 0$ such that $mh\leq n<m(h+1)$. We have $r_{Q}(n)\geq r_{Q}(mh)$, hence \[\frac{r_{Q}(n)}{n^2}\geq \frac{1+\Lambda_Q (mh)^2}{n^2}\geq \frac{1}{n^2}+\Lambda_Q\frac{(n-m)^2}{n^2}.\] Letting $n$ tend to infinity, we conclude that \begin{equation}\label{a2-}
	   a_2^{-}(kQ)=\liminf_{n\to\infty}\frac{r_{Q}(n)}{n^2}\geq \Lambda_Q. 
	\end{equation}
	The combination of (\ref{a2+}) and (\ref{a2-}) shows that $a_2^+(kQ)=a_2^-(kQ)=a_2(kQ)=\Lambda_Q$. 

	(b) By (a), it suffices to show that $\Lambda_Q\geq \frac{1}{2480}$ for every wild quiver $Q$. Let $Q'$ be a quiver obtained from $Q$ by removing one arrow. Then $q_{Q'}(\alpha)\geq q_Q(\alpha)$ for every $\alpha\in S_Q$. If $Q'$ is wild, then this implies that $\Lambda_{Q'}\leq\Lambda_Q$. Therefore, it suffices to prove $\Lambda_Q\geq \frac{1}{2480}$ in the case when the underlying graph of $Q$ is a minimal wild graph. By \Cref{minimaleffective}, $Q$ is effective and $\Lambda_Q=\lambda_Q$, hence we may compute $\Lambda_Q$ using the method of Lagrange multipliers, as explained in \Cref{critical}. We list the values of $\Lambda_Q=\lambda_Q$ that we have computed for the minimal wild graphs, following the notation given in \cite[Lecture 6, Subsection 6.7]{leit6}; another source is \cite[\S 2.4]{faenzi2013one}. The reader can find a picture of $\widetilde{\widetilde{E_8}}$ in the statement of \Cref{wild}(c), and pictures of $L_2, K_3$ and $U_5$ in the examples of \Cref{threefamilies}. The remaining wild graphs are displayed in \Cref{figure}.
	
	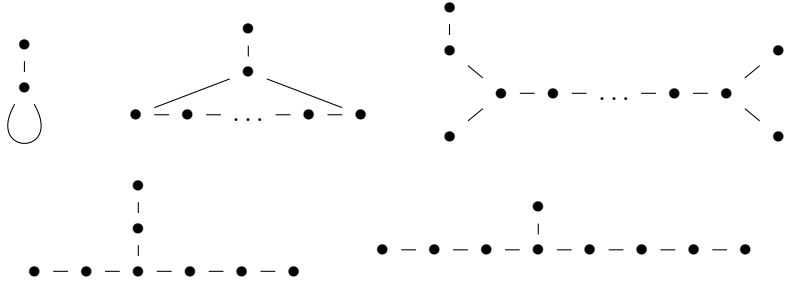
\begin{figure}
		\centering
	\[
		\begin{tikzcd}[row sep=tiny, column sep = tiny]
	 \bullet \arrow[d,-] \\
	 \bullet \arrow[out=-120,in=-60,-,loop]\\ 
	\end{tikzcd}\qquad
	\begin{tikzcd}[row sep=tiny, column sep = tiny]
	&& \bullet \arrow[d,-] \\
	&& \bullet \arrow[dll,-] \arrow[drr,-]  \\
	\bullet \arrow[r,-] & \bullet \arrow[r,-] & \dots\arrow[r,-] & \bullet \arrow[r,-] & \bullet
	\end{tikzcd}\qquad 
	\begin{tikzcd}[row sep=tiny, column sep = tiny]
	\bullet\arrow[d,-] \\
	\bullet \arrow[dr,-]&&&&&& \bullet\\
	& \bullet \arrow[r,-] & \bullet \arrow[r,-] &\dots \arrow[r,-] & \bullet \arrow[r,-] & \bullet \arrow[ur,-] \arrow[dr,-]\\    
	\bullet \arrow[ur,-] &&&&&& \bullet
	\end{tikzcd}
	\]
	\[
	\begin{tikzcd}[row sep=tiny, column sep = tiny]
	&& \bullet \arrow[d,-]  \\
	&& \bullet \arrow[d,-] \\
	\bullet \arrow[r,-]  & \bullet \arrow[r,-] & \bullet \arrow[r,-] & \bullet \arrow[r,-] & \bullet \arrow[r,-] & \bullet 
	\end{tikzcd}\qquad 
	\begin{tikzcd}[row sep=tiny, column sep = tiny]
	&&& \bullet \arrow[d,-]  \\
	\bullet \arrow[r,-] & \bullet \arrow[r,-] & \bullet \arrow[r,-] & \bullet \arrow[r,-] & \bullet \arrow[r,-] & \bullet \arrow[r,-] & \bullet \arrow[r,-] & \bullet 
	\end{tikzcd}
	\]
	\caption{The wild graphs $\widetilde{\widetilde{A_0}}$, $\widetilde{\widetilde{A_n}}$ (minimal for $1\leq n\leq 6$), $\widetilde{\widetilde{D_n}}$ (minimal for $4\leq n\leq 8$), $\widetilde{\widetilde{E_6}}$ and $\widetilde{\widetilde{E_7}}$. If one of these graphs has the index $j\geq 0$ in the name, it has $j+2$ vertices. \label{figure}} 
\end{figure}

\begin{multicols}{3} 
	\begin{itemize}\addtolength{\itemsep}{0.5\baselineskip}
		\item[--] $L_2$: $\lambda=1$,
		\item[--] $K_3$: $\lambda=\frac{1}{4}$,
		\item[--] $U_5$: $\lambda=\frac{1}{44}$,
		\item[--] $\widetilde{\widetilde{A_0}}$: $\lambda=\frac{1}{8}$,
		\item[--] $\widetilde{\widetilde{A_1}}$: $\lambda=\frac{1}{23}$,
		\item[--] $\widetilde{\widetilde{A_2}}$: $\lambda=\frac{1}{44}$,
		\item[--] $\widetilde{\widetilde{A_3}}$: $\lambda=\frac{1}{70}$,
		\item[--] $\widetilde{\widetilde{A_4}}$: $\lambda=\frac{1}{100}$,
		\item[--] $\widetilde{\widetilde{A_5}}$: $\lambda=\frac{1}{133}$,
		\item[--] $\widetilde{\widetilde{A_6}}$: $\lambda=\frac{1}{168}$,
		\item[--] $\widetilde{\widetilde{D_4}}$: $\lambda=\frac{1}{140}$,
		\item[--] $\widetilde{\widetilde{D_5}}$: $\lambda=\frac{1}{228}$,
		\item[--] $\widetilde{\widetilde{D_6}}$: $\lambda=\frac{1}{330}$,
		\item[--] $\widetilde{\widetilde{D_7}}$: $\lambda=\frac{1}{442}$,
		\item[--] $\widetilde{\widetilde{D_8}}$: $\lambda=\frac{1}{560}$,
		\item[--] $\widetilde{\widetilde{E_6}}$: $\lambda=\frac{1}{468}$,
		\item[--] $\widetilde{\widetilde{E_7}}$: $\lambda=\frac{1}{969}$,
		\item[--] $\widetilde{\widetilde{E_8}}$: $\lambda=\frac{1}{2480}$.
	\end{itemize}
\end{multicols}

The smallest value is $\frac{1}{2480}$, corresponding to $\widetilde{\widetilde{E_8}}$. The critical point is \[\alpha_{\widetilde{\widetilde{E_8}}}=\frac{1}{1240}(76,153,231,115,195,160,126,93,61,30).\]

(c) Let $Q$ be a wild quiver such that $a_2(kQ)=\Lambda_Q=\frac{1}{2480}$. By (b), this is the minimal possible value for $\Lambda_Q$. We may assume that $Q$ is connected. We must show that $Q$ is of type $\widetilde{\widetilde{E_8}}$.

We first claim that every subgraph of the underlying graph of $Q$ which is minimal wild must be of type $\widetilde{\widetilde{E_8}}$. Indeed, let $Q'$ be a subgraph of $Q$, and denote by $Q''$ the subquiver of $Q$ such that $Q''_0=Q'_0$. The effective subquivers of $Q''$ are also effective subquivers of $Q$, hence $\Lambda_{Q''}\leq \Lambda_Q$ by \Cref{maximum}. Since $Q'$ is obtained from $Q''$ by removal of some arrows, we have $q_{Q'}(\alpha)\geq q_{Q''}(\alpha)$ for every $\alpha\in S_{Q''}$. If $Q'$ is wild, then this implies that $\Lambda_{Q'}\leq\Lambda_{Q''}$. Therefore $\Lambda_{Q'}\leq \Lambda_Q$. If the graph of $Q'$ is minimal wild, by \Cref{maximum}, \Cref{minimaleffective} and the above list $Q'$ must be of type $\widetilde{\widetilde{E_8}}$, as claimed.

We now claim that $Q$ is obtained from a quiver of type $\widetilde{\widetilde{E_8}}$ by only adding arrows (and not vertices). To prove this, we may of course assume that $Q$ has no loops or multiple arrows.  By the previous claim, there exists in particular a wild subquiver $Q'$ of $Q$ which is obtained from a quiver of type $\widetilde{\widetilde{E_8}}$ by only adding arrows. We want to show that $Q_0=Q_0'$. If $Q_0\neq Q_0'$, we may pick a vertex $i\in Q_0\setminus Q_0'$ which is connected to $Q_0'$ by at least one arrow of $Q$. We let $Q''$ be the subquiver of $Q$ defined by the set of vertices $Q_0'\cup\set{i}$. Since $Q''$ is a subquiver of $Q$, by \Cref{maximum} we have $\Lambda_{Q''}\leq \Lambda_Q$, and since $\Lambda_Q$ is minimal we obtain that $\Lambda_{Q''}=\Lambda_Q=\frac{1}{2480}$. We now verify by a case by case analysis that $\Lambda_{Q''}>\frac{1}{2480}$, thus proving that $i\in Q_0\setminus Q_0'$ cannot exist. 

If two or more vertices of $Q'$ are connected to $i$ via an arrow, then $Q''$ contains a subquiver of type $\widetilde{\widetilde{A_h}}$, for some $0\leq h\leq 9$. One easily sees that this implies that $Q''$ (hence $Q$) contains a subquiver whose underlying graph is minimal wild of type $\widetilde{\widetilde{A_h}}$ for $0\leq h\leq 6$, $\widetilde{\widetilde{E_6}}$ or $\widetilde{\widetilde{E_7}}$. 
This contradicts the first claim. Therefore, $i$ is connected to exactly one $j\in Q'_0$. We now want to exclude this possibility.
\begin{itemize}
	\item[--] If $j=1$, then $Q$ contains a subquiver of type  $\widetilde{\widetilde{E_7}}$.
	\item[--] If $j=2$, then $Q$ contains a subquiver of type $\widetilde{\widetilde{D_5}}$.
	\item[--] If $j=3$, then $Q$ contains a subquiver of type $\widetilde{\widetilde{D_4}}$.	
	\item[--] If $j=4$, then $Q$ contains a subquiver of type $\widetilde{\widetilde{E_6}}$.
	\item[--] If $5\leq j\leq 8$, then $Q$ contains a subquiver of type $\widetilde{\widetilde{D_j}}$.
\end{itemize}
Therefore, the cases $1\leq j\leq 8$ are in contradiction with the first claim. The only remaining possibilities are $j=9$ and $j=10$. For $j=9,10$, the only subquivers of $Q''$ whose graphs are minimal wild are of type $\widetilde{\widetilde{E_8}}$ (note that $\widetilde{\widetilde{D_9}}$ is not minimal), and so we must use a different reasoning to exclude these two cases. If $j=9$, one finds that 
\[\alpha_{Q''}=\frac{1}{340}(16, 33, 51, 25, 45, 40, 36, 33, 31, 15, 15)\] and $\lambda_{Q''}=\frac{1}{680}$.
If $j=10$, then  
\[\alpha_{Q''}=\frac{1}{1639}(94, 190, 288, 143, 245, 204, 165, 128, 93, 60, 29)\] and $\lambda_{Q''}=\frac{1}{1639}$. In both cases, the coordinate of $\alpha_{Q''}$ corresponding to $i$ is the last one.
Thus, in each case $Q''$ is effective and $\lambda_{Q''}>\frac{1}{2480}$. By \Cref{maximum}, \[\Lambda_Q\geq \Lambda_{Q''}=\lambda_{Q''}>\frac{1}{2480}\] for $j=9,10$ too. This contradicts the assumptions, hence no such $i$ exists, and so $Q_0=Q_0'$, as claimed. We have shown that $Q$ is obtained from a quiver of type $\widetilde{\widetilde{E_8}}$ by only adding arrows, as claimed. In particular, $Q$ is a minimal wild quiver.

To conclude the proof, it is enough to show that the underlying graph of $Q$ is a minimal wild graph. If we remove one arrow from $Q$, we obtain a new quiver $Q'$ such that $q_{Q'}(\alpha)>q_Q(\alpha)$ for each $\alpha\in \mathring{S}_Q$. Since $Q$ has no proper wild subquivers, neither does $Q'$. If $Q'$ were wild, it would be a minimal wild quiver, hence by \Cref{maximum} and \Cref{minimaleffective} we would have $\Lambda_{Q'}=\lambda_{Q'}$. This would in turn imply \[\Lambda_{Q'}=\lambda_{Q'}=-q_{Q'}(\alpha_{Q'})<-q_{Q}(\alpha_{Q'})\leq \Lambda_Q.\] This is in contradiction with the minimality of $\Lambda_Q$, hence $Q'$ is non-wild. Thus every subgraph of $Q$ is the graph of a tame quiver, that is, the underlying graph of $Q$ is a minimal wild graph, as desired. 
\end{proof}

\begin{rmk}\label{bigo}
Let $Q$ be a wild quiver. The proof of \Cref{wild}(a) shows that $r_Q(n)=a_2(kQ)n^2+O(n)$ as $n\to\infty$.
\end{rmk}

\section{Examples}\label{threefamilies}

\Cref{maximum} gives a simple algorithm to determine $\Lambda_Q$ for a given quiver $Q$. For each wild subquiver $Q'$ of $Q$, one determines $\alpha_{Q'}$ and $\lambda_{Q'}$ by solving a system of linear equations (for minimal wild subquivers, one may use the list in the proof of \Cref{wild}). By inspection of the $\alpha_{Q'}$ and $\lambda_{Q'}$, one lists the connected effective subquivers of $Q$. Then $\Lambda_Q$ coincides with the maximum value of $\lambda_{Q'}$ among the subquivers in the list. As an example, we determine the coefficients $a_0,a_1,a_2$ explicitly for some families of quivers.

\begin{example}
	Let $K_r$ be an $r$-Kronecker quiver, 	
	\[
	\begin{tikzcd}
	1
	\arrow[r,-, draw=none, "\raisebox{+1.5ex}{\vdots}" description]
	\arrow[r,-, bend left]
	\arrow[r,-, bend right, swap]
	&
	2
	\end{tikzcd}\]
	with an arbitrary orientation of the arrows. The underlying graph of $K_1$ is a Dynkin diagram of type $A_2$. The quiver $K_2$ is tame of type $\widetilde{A}_2$, and the null root is given by $\delta=(1,1)$. If $r\geq 3$, $K_r$ is a minimal wild quiver. One may easily compute that $\alpha_{K_r}=(\frac{1}{2},\frac{1}{2})$, and $\Lambda_{K_r}=\lambda_{K_r}=\frac{r-2}{4}$. Using \Cref{finrep}(b), \Cref{tameed}, \Cref{wild} and \Cref{bigo}, we obtain
\begin{align*}
r_{K_r}(n)=\begin{cases}
0, &\text{ if $r=1$,}\\
\floor{\dfrac{n}{2}}, &\text{ if $r=2$,}\\
\dfrac{r-2}{4}n^2+O(n), &\text{ if $r\geq 3$.}\\
\end{cases}
\end{align*}
\end{example}

\begin{example}	
	Let $L_r$ be the $r$-loop quiver. It is the quiver with one vertex and $r$ arrows, here depicted for $r=4$:
		
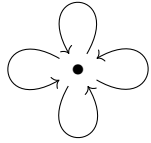
\[
	\begin{tikzcd}
	\bullet \arrow[out=60,in=120,loop,distance=3em] \arrow[out=-30,in=30,loop,distance=3em] \arrow[out=-120,in=-60,loop,distance=3em] \arrow[out=-210,in=-150,loop,distance=3em]
	\end{tikzcd}
	\]
	
	The quiver $L_1$ is tame of type $\widetilde{A}_1$, with null root $\delta=(1)$. If $r\geq 2$, $L_r$ is wild, $\alpha_{L_r}=(1)$ and $\Lambda_{L_r}=\lambda_{L_r}=r-1$. Thus
	
	\begin{align*}
	r_{L_r}(n)=\begin{cases}
	n, &\text{ if $r=1$,}\\
	(r-1)n^2+O(n), &\text{ if $r\geq 2$.}\\
	\end{cases}
	\end{align*}
\end{example}

\begin{example}\label{starshaped}
	
	Let $U_r$ be an $r$-starshaped quiver, that is, a quiver with vertices $0,1,\dots, r$ and such that for every $i=1,\dots,r$ there exists exactly one arrow connecting $i$ and $0$, and these are the only arrows. The orientation of arrows is arbitrary. Here is the picture for $r=4$:
	
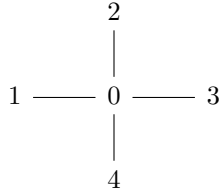
\[
	\begin{tikzcd}
	& 2 \arrow[d,-] \\
	1 \arrow[r,-] & 0 & 3 \arrow[l,-] \\
	& 4 \arrow[u,-]
	\end{tikzcd}
	\]
	If $r=1,2,3$, $U_r$ is of finite representation type, hence $r_{U_r}(n)=0$. The quiver $U_4$ is tame of type $\widetilde{D}_4$, and has a null root $\delta=(2,1,1,1,1)$, so $r_{U_4}(n)=\floor{n/6}$. If $r\geq 5$, then $U_r$ is wild. Moreover, $\alpha_{U_r}$ is the solution of the following system of linear equations:
\[
\begin{cases}
\sum\alpha_i=1\\
2\alpha_0-\sum_{i\neq 0}\alpha_i=-2\lambda_{U_r}\\
2\alpha_i-\alpha_0=-2\lambda_{U_r} \text{ for each $i\neq 0$.}
\end{cases}
\]
Note that these equations imply that $\alpha_i=\alpha_1$ for each $i\neq 0$. The solution of this linear system is $\alpha_{U_r}=\frac{1}{4r+2}(r+2,3,3,\dots,3)$, and $\lambda_{U_r}=\frac{r-4}{8r+4}$. In particular, for $r\geq 5$, $U_r$ is always effective for $r\geq 5$, and the sequence $\lambda_{U_r}$ is strictly increasing in $r$. Since any subquiver of $U_r$ is itself of the form $U_{r'}$, for some $r'\leq r$, we deduce that $\Lambda_{U_r}=\lambda_{U_r}=\frac{r-4}{8r+4}$. Therefore
\begin{align*}
r_{K_r}(n)=\begin{cases}
0, &\text{ if $r\leq 3$,}\\
\floor{\dfrac{n}{6}}, &\text{ if $r=4$,}\\
\dfrac{r-4}{8r+4}n^2+O(n), &\text{ if $r\geq 5$.}\\ 
\end{cases}
\end{align*}
Notice that $\Lambda_{U_r}\leq \frac{1}{8}$ for any $r\geq 5$.
\end{example}

\section{Appendix}
The purpose of this Appendix is to prove \Cref{uppertame}, which is used in the proof of \Cref{tameed}. 

\begin{lemma}\label{sep}
	Let $K$ be a separably closed field, and let $M$ be an indecomposable $A_K$-module. Then $M_{\cl{K}}$ is an indecomposable $A_{\cl{K}}$-module.
\end{lemma}

\begin{proof}
	The result is trivial if $\on{char}K=0$, so we may assume that $\on{char}K=p>0$. We must prove that $M_{\cl{K}}$ is indecomposable as an $A_{\cl{K}}$-module. Since $K$ is separably closed, $L:=\on{End}(M)/j(\on{End}(M))$ is a field. Let $\phi$ be an idempotent in $L\otimes_K\cl{K}$. We may write $\phi=\sum \phi_i\otimes\lambda_i$ for some $\phi_i\in L$ and some $\lambda_i\in \cl{K}$. The extension $\cl{K}/K$ being purely inseparable, there exists a positive integer $n$ such that $\lambda_i^{p^n}\in K$ for every $i$. Since $\phi$ is idempotent, \[\phi=\phi^{p^n}=\sum \phi_i^{p^n}\otimes\lambda_i^{p^n}=\sum (\lambda_i\phi_i)^{p^n}\otimes 1\] belongs to $L$. Since $L$ is a field, we obtain $\phi=0,1$. This proves that $L\otimes_K\cl{K}$ is a local ring. Using the inclusion $j(\on{End}(M))\otimes_K\cl{K}\c j(\on{End}(M_{\cl{K}}))$, it follows that $\on{End}(M_{\cl{K}})/j(\on{End}(M_{\cl{K}}))$ is a local ring too, which means that $M_{\cl{K}}$ is indecomposable.
\end{proof}

For the proof of \Cref{uppertame} we closely follow \cite{donovanfreislich}. Another reference for the classification of indecomposable representations of tame quivers is \cite[Chapter XIII]{simson2007elements}.

\begin{proof}[Proof of \Cref{uppertame}]
	For every arrow $a:i\to j$ of $Q$, we denote by $\phi_a:M_i\to M_j$ the associated $K$-linear map of $M$. Assume that $Q$ is a tame quiver of type $\widetilde{A}_n$ (so the underlying graph of $Q$ is a cycle with $n+1$ vertices). Its null root is $\delta=(1,\dots,1)$. The case $n=0$ has already been treated in \Cref{basicexample}. Suppose that $Q$ has a cyclic orientation. Let $N$ be an indecomposable summand of $M_{\cl{K}}$, and $\psi_a:N_i\to N_j$ be the linear map associated to the arrow $a:i\to j$. By \cite[Theorem 7.6]{kirillov}, either (i) all the $\phi_a$ can be represented by matrices containing only $0$ and $1$, or (ii) all the $\phi_a$ but at most one are isomorphisms. By \Cref{sep}, every indecomposable summand of $M_{\cl{K}}$ is already defined over $K^{\on{sep}}$. The Galois group $\on{Gal}(K^{\on{sep}}/K)$ acts transitively on the isomorphism classes of indecomposable representations of $M_{K^{\on{sep}}}$. Thus, when one of the above is true for $N$, it is also true for every other indecomposable summand of $M_{\cl{K}}$ (we will use this reasoning multiple times during this proof). In case (i), by Noether-Deuring's Theorem, $M$ may be represented by matrices with entries in $\set{0,1}$. In case (ii), after fixing bases for the vector spaces $M_i$, we may assume that all the $\phi_a$ but at most one are represented by the identity matrix, and we are reduced to the case $n=0$. 
	
	Assume now that $Q$ is of type $\widetilde{A}_{2r+1}$, oriented in such a way that every even vertex is a sink, and every odd vertex is a source. For example, this is the orientation that we are considering on $\widetilde{A}_5$ (so $r=2$).	
	
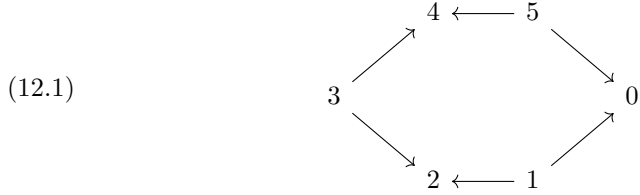
\begin{equation}\label{a5}
	\begin{tikzcd}
	& 4 & 5 \arrow[l] \arrow[dr] \\ 
	3 \arrow[ur] \arrow[dr] &&& 0 \\
	& 2	& 1 \arrow[l] \arrow[ur] 
	\end{tikzcd}
	\end{equation}\noindent
	If $r=0$, we have the Kronecker quiver, whose indecomposable $K$-representations are well-known for both orientations; see for example \cite[Theorem 3.6]{burban2012two}. Let now $r\geq 1$. Consider the base change $M_{\cl{K}}$, and denote by $N$ an indecomposable summand of $M_{\cl{K}}$, with linear maps $\psi_a:N_i\to N_j$. By \cite[Lemma 2.6.5]{donovanfreislich}, each $\psi_a$ is an isomorphism, with the exception of at most two. By applying \Cref{sep} as in the first paragraph, we deduce that all arrows of $M$ but at most two are represented by linear isomorphisms. Identifying vertices via these isomorphisms, we are reduced to the cases $\widetilde{A}_0$ or $\widetilde{A}_1$, which have already been handled. Consider now the case when $Q$ is of type $\widetilde{A}_n$, where $n$ is not necessarily odd and the orientation is acyclic but otherwise arbitrary. Adding arrows to $Q$ if necessary, we may identify $M$ with an indecomposable representation $M'$ of a quiver $Q'$ of type $\widetilde{A}_{2r+1}$, having the orientation given in (\ref{a5}) (i.e. every even vertex is a sink and every odd vertex is a source), for a suitable $r$. Of course, we require the new arrows to be represented by isomorphisms. It follows that if $M'$ may be defined using $0,1,a_1,\dots,a_m$, the same is true for $M$. This concludes the proof for quivers of type $\widetilde{A}_n$. 
	
	The case when $Q$ is of type $\widetilde{D}_n$ can be proved along similar lines. If $n=4$ and $Q$ has the orientation 
	
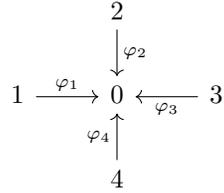
\[
	\begin{tikzcd}
	& 2 \arrow[d,"\phi_2"] \\
	1 \arrow[r,"\phi_1"] & 0 & 3 \arrow[l,"\phi_3"] \\
	& 4 \arrow[u,"\phi_4"]
	\end{tikzcd}
	\]\noindent
	the indecomposable representations of $Q$ have been classified: see \cite{gelfand1970problems} for the original proof over algebraically closed fields, and \cite{medina2004four} for an elementary proof over arbitrary fields. Recall that the null root of $Q$ is $\delta=(2,1,1,1,1)$. We record here the $m\delta$-dimensional family consisting of all the $K$-representations of $Q$ that are not defined over the prime field of $K$ (see \cite[Appendix]{medina2004four}). 
	\begin{equation*}
	\phi_1=\begin{pmatrix} I_{m\times m}\\ 0 \end{pmatrix},\qquad \phi_2=\begin{pmatrix}0 \\ I_{m\times m}\end{pmatrix}, \qquad \phi_3=\begin{pmatrix}I_{m\times m} \\ I_{m\times m}\end{pmatrix}, \qquad \phi_4=\begin{pmatrix}A \\ I_{m\times m}\end{pmatrix}.
	\end{equation*}
	Here each of the eight blocks is a square matrix of size $m$, and $A$ is a square matrix of size $m$ in rational canonical form. If $M$ does not belong to this family, $M$ may be defined using only $0$ and $1$. On the other hand, if $M$ belongs to the family, $M$ may be defined using only $0,1,a_1,\dots,a_m$, where the $a_i$ are the coefficients of the last column of $A$.
	
	Assume now that $Q$ is of type $\widetilde{D}_n$, where $n>4$. By suitably adding arrows, similarly to what we have done in type $\widetilde{A}_n$, it suffices to consider the case when $n=2r+4\geq 6$ is even, and with the following orientation of arrows (here $r=3$): 
	\[
	\begin{tikzcd}
	& 2  \arrow[d] &&&& 8 \arrow[d]\\
	1 \arrow[r]
	& 3 
	& 4 \arrow[l] \arrow[r] 
	& 5 
	& 6 \arrow[l] \arrow[r] 
	& 7
	& 9. \arrow[l]
	\end{tikzcd}
	\]\noindent
	In other terms, the sinks are exactly the odd vertices different from $1$ and $2r+3$, and every other vertex is a source.
	
	By \cite[Lemma 3.8.5]{donovanfreislich}, if $N$ is an indecomposable summand of $M_{\cl{K}}$, either (i) $N$ can be defined by matrices with entries only $0$ and $1$, or (ii) all but at most two of the maps	
	\[N_1\oplus N_2\to N_3, N_4\to N_3,\dots, N_{2r}\to N_{2r+1}, N_{2r+2}\oplus N_{2r+3}\to N_{2r+1}\]
	must be isomorphisms. Applying \Cref{sep} as in the first paragraph, we see that in case (i) $M$ is defined by matrices consisting only of $0$ and $1$, and in case (ii) all but at most two of the linear maps \[M_1\oplus M_2\to M_3, M_4\to M_3,\dots, M_{2r}\to M_{2r+1}, M_{2r+2}\oplus M_{2r+3}\to M_{2r+1}\] are isomorphisms. Now, if one of the two maps \[M_1\oplus M_2\to M_3,\qquad M_{2r+3}\oplus M_{2r+4}\to M_{2r+2}\] is an isomorphism, then $M$ comes from a representation of a Dynkin quiver of type $D_{2r+2}$. If neither of these two arrows is represented by an isomorphism, then $M$ comes from a representation of a quiver of type $\widetilde{D}_4$. Since the underlying graph of $Q$ is a tree, by \cite[Lemma 3.6]{kirillov} any two orientations of $Q$ may be obtained one from the other via reflection functors. This proves the claim for quivers of type $\widetilde{D}_n$.
	
	To complete the proof of \Cref{uppertame}, only type $\widetilde{E}$ is left. In \cite{donovanfreislich}, the classification in type $\widetilde{E}$ is deduced from that of type $\widetilde{A}$ and $\widetilde{D}$ by means of certain functorial constructions. The arguments of \cite{donovanfreislich} work over an arbitrary field, as the authors say in \cite[\S 1.1]{donovanfreislich}. However, some of the references that they quote need to be modified; we now explain how.
	
	Assume first that $Q$ is a tame quiver of type $\widetilde{E}_6$, with the following orientation:
	\[
	\begin{tikzcd}
	&& 5 \arrow[d]\\ 
	&& 4 \arrow[d]\\
	3 \arrow[r]
	& 2 \arrow[r]
	& 1 
	& 6 \arrow[l]
	& 7. \arrow[l]
	\end{tikzcd}
	\]\noindent
	
	With this ordering, the null root of $Q$ is $\delta=(3,2,1,2,1,2,1)$. Let $Q'$ be the quiver (\ref{a5}), and let $\delta'$ be its null root. We construct a functor $F$ from the category of $K$-representations of $Q'$ of dimension $m\delta'$ to the category of $m\delta$-dimensional $K$-representations of $Q$ as follows. Let $N$ be a $K$-representation of $Q$, of dimension vector $m\delta'$, and denote by $N_0,\dots,N_5$ the vector spaces of $N$. Then $F(N)$ is given by the vector spaces (following the ordering in the figure): \[N_0\oplus N_2\oplus N_4,\quad N_0\oplus N_2,\quad N_1,\quad N_2\oplus N_4,\quad N_3,\quad N_4\oplus N_0,\quad N_5,\] and by linear maps defined in an obvious way using those of $N$. The functor $F$ is denoted by $S_6$ in \cite[4.5]{donovanfreislich}.
	
	If $M$ may not be defined using only $0$ and $1$, then $M$ belongs to the essential image of $F$. The proof of this fact is given in \cite[Theorem 4.8.1]{donovanfreislich} in the case when $K$ is algebraically closed. This argument is based on elementary linear algebra and works over an arbitrary field; see \cite[\S 1.1]{donovanfreislich}. The only step that requires further justification is the assertion that the category of regular $K$-representations of $Q$ is abelian. If $K$ is algebraically closed, this is proved in \cite[Proposition 4.7.1]{donovanfreislich}. For the case, where $K$ is an arbitrary field, we refer the reader to \cite[Proposition 3.2]{dlab1976indecomposable} or \cite[\S 4.1]{ringel1976representations}; see also \cite[\S 2.4]{Wolf:1188722} or the Introduction to \cite{dlab1976indecomposable}. It follows that $M$ comes from an $m\delta'$-dimensional representation $M'$ of a quiver $Q'$ of type $\widetilde{D}_4$. We know that $M'$ may be defined using $0,1,a_1,\dots,a_m$, for some $a_i\in K$, thus the same is true for $M$. By \cite[Lemma 3.6]{kirillov}, applying the reflection functors, this proves \Cref{uppertame} for every other orientation of $\widetilde{E}_6$.	
	
	The proof for $Q$ of type $\widetilde{E}_7$ or $\widetilde{E}_8$ is entirely analogous. The indecomposable representations of $\widetilde{E}_7$ not defined over the prime field of $K$ may be obtained from representations of $\widetilde{E}_6$, and those of $\widetilde{E}_8$ may be obtained from those of $\widetilde{E}_7$. The fact that the category of regular $K$-representations of $Q$ is abelian is proved in \cite[Proposition 5.7.1 and Proposition 6.7.1]{donovanfreislich} for an algebraically closed field $K$, and in \cite[Proposition 3.2]{dlab1976indecomposable} and \cite[\S 4.1]{ringel1976representations} for an arbitrary $K$. The rest of the proof is based on elementary linear algebra, and may be carried out over an arbitrary field.
\end{proof}

\section*{Acknowledgements}
I am very grateful to Angelo Vistoli for proposing to me this topic of research, and for making me aware of the methods in \cite{biswasdhillonhoffmann}, and to my advisor Zinovy Reichstein for his help and suggestions. I thank Roberto Pirisi and Mattia Talpo for helpful comments, and Ajneet Dhillon and Norbert Hoffmann for useful correspondence. I thank the referees for many useful suggestions.

\end{document}